\documentclass[a4paper, 11pt]{amsart}   
\usepackage{mathptmx, amssymb,amscd,latexsym, eulervm}   
\usepackage{amsmath}
\usepackage{amsthm}
\usepackage{mathdots}
\usepackage[pagebackref,colorlinks=true,linkcolor=blue,urlcolor=blue]{hyperref}
\usepackage{color}
\usepackage[onehalfspacing]{setspace}
\usepackage{comment}
\usepackage{tabularx}
\usepackage{amsfonts}
\usepackage{paralist}
\usepackage{aliascnt}
\usepackage[initials, lite]{amsrefs}
\usepackage{amscd}
\usepackage{blkarray}
\usepackage{mathbbol}
\usepackage{setspace}
\usepackage[inner=2.5cm,outer=2.5cm, bottom=3.2cm]{geometry}
\usepackage{tikz, tikz-cd}
\usepackage{calligra,mathrsfs}

\usepackage{tikz}
\usetikzlibrary{matrix}
\usetikzlibrary{arrows,calc}
\allowdisplaybreaks

\BibSpec{collection.article}{%
	+{}  {\PrintAuthors}                {author}
	+{,} { \textit}                     {title}
	+{.} { }                            {part}
	+{:} { \textit}                     {subtitle}
	+{,} { \PrintContributions}         {contribution}
	+{,} { \PrintConference}            {conference}
	+{}  {\PrintBook}                   {book}
	+{,} { }                            {booktitle}
	+{,} { }                            {series}
	+{, vol.} { }                            {volume}
	+{,} { }                            {publisher}
	+{,} { \PrintDateB}                 {date}
	+{,} { pp.~}                        {pages}
	+{,} { }                            {status}
	+{,} { \PrintDOI}                   {doi}
	+{,} { available at \eprint}        {eprint}
	+{}  { \parenthesize}               {language}
	+{}  { \PrintTranslation}           {translation}
	+{;} { \PrintReprint}               {reprint}
	+{.} { }                            {note}
	+{.} {}                             {transition}
	+{}  {\SentenceSpace \PrintReviews} {review}
}

\newcommand{\kk}{\mathbb{k}}

\newcommand{\GG}{\mathcal{G}}

\newcommand{\NN}{\normalfont\mathbb{N}}
\newcommand{\ZZ}{{\normalfont\mathbb{Z}}}

\newcommand{\Z}{{\normalfont\mathbb{Z}}}

\newcommand{\PP}{\normalfont\mathbb{P}}

\newcommand{\xx}{\normalfont\mathbf{x}}

\newcommand{\mm}{{\normalfont\mathfrak{m}}}

\newcommand{\QQ}{\mathbb{Q}}
\newcommand{\pp}{\mathfrak{p}}

\newcommand{\bb}{\mathfrak{b}}
\newcommand{\qqq}{\mathfrak{q}}

\newcommand{\rank}{\normalfont\text{rank}}

\newcommand{\Ext}{\normalfont\text{Ext}}

\newcommand{\Ker}{\normalfont\text{Ker}}
\newcommand{\Coker}{\normalfont\text{Coker}}
\newcommand{\Quot}{{\normalfont\text{Quot}}}
\newcommand{\Fib}{{\normalfont\text{Fib}}}
\newcommand{\IM}{\normalfont\text{Im}}

\newcommand{\Hom}{{\normalfont\text{Hom}}}

\newcommand{\grHom}{{}^*\Hom}
\newcommand{\Fitt}{\normalfont\text{Fitt}}

\newcommand{\OO}{\mathcal{O}}

\newcommand{\FF}{\mathcal{F}}
\newcommand{\HL}{\normalfont\text{H}_{\mm}}
\newcommand{\HH}{\normalfont\text{H}}
\newcommand{\bBB}{\normalfont\text{B}}
\newcommand{\zZZ}{\normalfont\text{Z}}
\newcommand{\cCC}{\normalfont\text{C}}

\newcommand{\iniTerm}{\normalfont\text{in}}

\newcommand{\Proj}{\normalfont\text{Proj}}
\newcommand{\Spec}{{\normalfont\text{Spec}}}
\newcommand{\ACM}{{\normalfont\text{ACM}}}
\newcommand{\AG}{{\normalfont\text{AG}}}

\newcommand{\FExt}{{\normalfont\text{FExt}}}
\newcommand{\FDir}{{\normalfont\text{FDir}}}
\newcommand{\FLoc}{{\normalfont\text{FLoc}}}

\newcommand{\Hilb}{\normalfont\text{Hilb}}
\newcommand{\GL}{\normalfont\text{GL}}

\newtheorem{theorem}{Theorem}[section]

\newtheorem{headthm}{Theorem}

\newaliascnt{headcor}{headthm}

\aliascntresetthe{headcor}

\newaliascnt{headconj}{headthm}

\aliascntresetthe{headconj}

\newaliascnt{corollary}{theorem}
\newtheorem{corollary}[corollary]{Corollary}
\aliascntresetthe{corollary}

\newaliascnt{claim}{theorem}

\aliascntresetthe{claim}

\newaliascnt{lemma}{theorem}
\newtheorem{lemma}[lemma]{Lemma}
\aliascntresetthe{lemma}

\newaliascnt{conjecture}{theorem}

\aliascntresetthe{conjecture}

\newaliascnt{proposition}{theorem}
\newtheorem{proposition}[proposition]{Proposition}
\aliascntresetthe{proposition}

\theoremstyle{definition}
\newaliascnt{definition}{theorem}
\newtheorem{definition}[definition]{Definition}
\aliascntresetthe{definition}

\newaliascnt{notation}{theorem}
\newtheorem{notation}[notation]{Notation}
\aliascntresetthe{notation}

\newaliascnt{example}{theorem}
\newtheorem{example}[example]{Example}
\aliascntresetthe{example}

\newaliascnt{examples}{theorem}

\aliascntresetthe{examples}

\newaliascnt{remark}{theorem}
\newtheorem{remark}[remark]{Remark}
\aliascntresetthe{remark}

\newaliascnt{question}{theorem}

\aliascntresetthe{question}

\newaliascnt{questions}{theorem}

\aliascntresetthe{questions}

\newaliascnt{problem}{theorem}

\aliascntresetthe{problem}

\newaliascnt{construction}{theorem}

\aliascntresetthe{construction}

\newaliascnt{setup}{theorem}
\newtheorem{setup}[setup]{Setup}
\aliascntresetthe{setup}

\newaliascnt{algorithm}{theorem}

\aliascntresetthe{algorithm}

\newaliascnt{observation}{theorem}

\aliascntresetthe{observation}

\newaliascnt{defprop}{theorem}

\aliascntresetthe{defprop}

\DeclareFontFamily{OT1}{pzc}{}
\DeclareFontShape{OT1}{pzc}{m}{it}{<-> s * [1.100] pzcmi7t}{}
\DeclareMathAlphabet{\mathchanc}{OT1}{pzc}{m}{it}

\DeclareMathOperator{\fFib}{\mathchanc{Fib}}
\DeclareMathOperator{\fHilb}{\mathchanc{Hilb}}
\DeclareMathOperator{\fQuot}{\mathchanc{Quot}}
\DeclareMathOperator{\fF}{\mathchanc{F}}
\DeclareMathOperator{\fFExt}{\mathchanc{FExt}}
\DeclareMathOperator{\fFLoc}{\mathchanc{FLoc}}
\DeclareMathOperator{\fFDir}{\mathchanc{FDir}}
\DeclareMathOperator{\fACM}{\mathchanc{ACM}}
\DeclareMathOperator{\fAG}{\mathchanc{AG}}

\def\equationautorefname~#1\null{(#1)\null}
\def\sectionautorefname~#1\null{Section #1\null}
\def\subsectionautorefname~#1\null{\S #1\null}

\newcommand{\ritvik}[1]{{\color{purple} \sf $\clubsuit$ Ritvik: [#1]}}

\begin{document}

	\title{The fiber-full scheme}
	
	\author{Yairon Cid-Ruiz}
	\address{Department of Mathematics, North Carolina State University, Raleigh, NC, 27695, USA}
	\email{ycidrui@ncsu.edu}
	
	\author{Ritvik Ramkumar}
	\address{Department of Mathematics, University of Notre Dame, South Bend, IN, 46556, USA}
	\email{rramkuma@nd.edu}
	

	\date{\today}
	
	\begin{abstract}
		We introduce the fiber-full scheme which can be seen as the parameter space that generalizes the Hilbert and Quot schemes by controlling the entire cohomological data.
		Let $f:X \subset \PP_S^r \rightarrow S$ be a projective morphism and $\mathbf{h} = (h_0,\ldots,h_r) : \ZZ^{r+1} \rightarrow \NN^{r+1}$ be a fixed tuple of functions.
		The fiber-full scheme $\Fib_{\FF/X/S}^\mathbf{h}$ is a fine moduli space parametrizing all quotients $\GG$ of a fixed coherent sheaf $\FF$ on $X$  such that $R^i{{f}_*}\left(\GG(\nu)\right)$ is a locally free $\OO_S$-module of rank equal to $h_i(\nu)$.
		In other words, the fiber-full scheme controls the dimension of all cohomologies of all possible twistings, instead of just the Hilbert polynomial.
		We show that the fiber-full scheme is a quasi-projective $S$-scheme and a locally closed subscheme of its corresponding Quot scheme.
		In the context of applications, we demonstrate that the fiber-full scheme provides the natural parameter space for arithmetically Cohen-Macaulay and arithmetically Gorenstein schemes with fixed cohomological data, and for square-free Gr\"obner degenerations.
	\end{abstract}

	\maketitle

\section{Introduction}	 \label{section_introduction}

The Hilbert and Quot schemes, introduced by Grothendieck \cite{GROTHENDIECK_HILB}, are among the most important constructions in algebraic geometry. 
The Quot scheme $\Quot_{\FF/X/S}^P$ is a fine moduli space parametrizing all quotients of a fixed coherent sheaf $\FF$ on a projective morphism $f:X \subset \PP_S^r \rightarrow S$ that are flat over a base scheme $S$ and have a fixed Hilbert polynomial $P(m) \in \QQ[m]$ along all fibers; the Hilbert scheme is a special case with $\FF = \OO_X$.
These schemes can be used to construct other important moduli spaces such as the moduli of curves and the Picard scheme \cite{GIT, PICARD}. 
They are used to study deformations of curves which is an essential part of birational geometry \cite{MORI_AMPLE, KOLLAR_MORI}. 
The Hilbert schemes of points on a singular plane curve and a surface are related to knot invariants \cite{HILB_KNOT_1, HILB_KNOT_2}.
 The Hilbert scheme of points on a surface appears in representation theory \cite{NAKAJIMA}, in combinatorics \cite{HAIMAIN} and, if the surface is a K3-surface, this Hilbert scheme provides an important class of examples of hyperk\"ahler varieties \cite{BEAUVILLE}. 
The main goal of this paper is to introduce a far-reaching generalization of the Hilbert and Quot schemes that controls all the cohomological data of the quotients of the coherent sheaf $\FF$ instead of just the Hilbert polynomial.

\medskip

We start with the classical example of the Hilbert scheme compactification of the space of twisted cubics that was studied by  Piene and  Schlessinger \cite{TWISTED_CUBIC}.
The motivating example below shows how this well-studied Hilbert scheme decomposes into locally closed subschemes that have constant cohomological data.

\begin{example}[\autoref{TWISTED_CUBIC}]
	\label{exmp_intro}
	In \cite{TWISTED_CUBIC}, it was shown that $\Hilb^{3m+1}_{\PP^3_{\kk}} = H \cup H'$ is a union of two smooth irreducible components such that the general member of $H$ parametrizes a twisted cubic, and the general member of $H'$ parametrizes a plane cubic union an isolated point. It is also known that $H - H \cap H'$ is the locus of arithmetically Cohen-Macaulay curves of degree $3$ and genus $0$.
	We then have a decomposition 
	$$
	\Hilb^{3m+1}_{\PP^3_{\kk}}  =  (H - H \cap H') \sqcup H'.
	$$
	Furthermore, one can show that the functions
	$$
	h_X^i : \ZZ \rightarrow \NN, \quad \nu \mapsto \dim_{\kk}\left(\HH^i(X, \OO_X(\nu))\right)
	$$
	are the same for any element $[X] \in H - H \cap H'$ and the same for any element $[X] \in H'$ (for an explicit computation, see \autoref{TWISTED_CUBIC}).
	It then follows that $\Hilb^{3m+1}_{\PP^3_{\kk}}$ can be decomposed into two locally closed subschemes where the cohomological functions 	$h_X^i$
	are constant.
	It should also be noted that one might be quite interested in studying $H - H \cap H'$ as it gives all the closed subschemes of $\PP_\kk^3$ with the same cohomological data as that of a twisted cubic.
\end{example}

As presented below, 
the scheme we introduce allows us to provide a unified and systematic treatment of the decomposition seen in \autoref{exmp_intro}.
Let $S$ be a locally Noetherian scheme, $f : X \subset \PP_S^r \rightarrow S$ be a projective morphism and $\FF$ be a coherent sheaf on $X$.
We follow Grothendieck's  general idea \cite{GROTHENDIECK_HILB} of considering a contravariant functor whose representing scheme (if it exists) is the parameter space one is interested in.

\medskip

We define the \emph{fiber-full functor} which for an $S$-scheme $T$ parametrizes all coherent quotients $\FF_T \twoheadrightarrow \GG$ such that all the higher direct images of $\GG$ and its twistings are locally free over $T$.
More precisely, for any (locally Noetherian) $S$-scheme $T$ we define
$$
\fFib_{\FF/X/S}(T) \;=\; \left\lbrace \text{coherent quotient $\FF_T \twoheadrightarrow \GG$} \;
\begin{array}{|l}
	R^i{{f_{(T)}}_*}\left(\GG(\nu)\right) \text{ is locally free over $T$}\\
	\text{for all $0 \le i \le r, \nu \in \ZZ$}
\end{array}
\right\rbrace,
$$
where $\FF_T$ is the pull-back sheaf on $X_T = X \times_S T$, $f_{(T)} : X_T \subset \PP_T ^r \rightarrow T$ is the base change morphism $f_{(T)} = f \times_S T$, and $\GG(\nu) = \GG \otimes \big(\OO_{\PP_T^r}(1)\big)^{\otimes \nu}$.
We have that 
$$
\fFib_{\FF/X/S} \;: \; (\text{Sch}/S)^\text{opp} \rightarrow (\text{Sets})
$$ 
is a contravariant functor from the category of (locally Noetherian) $S$-schemes to the category of sets (see \autoref{lem_base_change_Fib}).
We stratify this functor in terms of ``Hilbert functions'' for all the cohomologies.
Let $\mathbf{h} = (h_0,\ldots,h_r) : \ZZ^{r+1} \rightarrow \NN^{r+1}$ be a tuple of functions.
Then we define the following functor depending on $\mathbf{h}$: 
$$
\fFib^{\mathbf{h}}_{\FF/X/S}(T) \;=\; 
\left\lbrace \GG \in \fFib_{\FF/X/S}(T) \;\;
\begin{array}{|l}
	\text{$\dim_{\kappa(t)}\left(\HH^i\left(X_t,\GG_t(\nu)\right)\right) = h_i(\nu)$}\\
	\text{for all $0 \le i \le r, \nu \in \ZZ, t \in T$}
\end{array}
\right\rbrace,
$$
where $\kappa(t)$ denotes the residue field of the point $t \in T$, $X_t = X_T \times_T \Spec(\kappa(t))$ is the fiber over $t \in T$, and $\GG_t$ is the pull-back sheaf on $X_t$.
The idea of this functor is to measure the dimension of \emph{all cohomologies of all possible twistings}. 
 We easily obtain the  stratification 
 $
 \fFib_{\FF/X/S}(T) \;=\; \bigsqcup_{\mathbf{h} : \ZZ^{r+1} \rightarrow \NN^{r+1}} \fFib^{\mathbf{h}}_{\FF/X/S}(T)
 $
 when $T$ is connected, and so it follows that $\fFib_{\FF/X/S}(T)$ is a representable functor if all the functors $\fFib^{\mathbf{h}}_{\FF/X/S}(T)$ are representable.
 When $\FF = \OO_X$, we simplify the notation by writing $\fFib^{\mathbf{h}}_{X/S}$ instead of $\fFib^{\mathbf{h}}_{\OO_X/X/S}$.
 
 \medskip
 
 For any numerical polynomial $P \in \QQ[m]$,  we have Grothendieck's definition of the Quot functor 
 $$
 \fQuot_{\FF/X/S}^P : (\text{Sch}/S)^\text{opp} \rightarrow (\text{Sets})
 $$ which for an $S$-scheme $T$ parametrizes all coherent quotients $\FF_T \twoheadrightarrow \GG$ that are flat over $T$ and have Hilbert polynomial equal to $P$ along all fibers.
 The Hilbert functor $\fHilb_{\FF/X/S}^P$ is the special case of $\fQuot_{\FF/X/S}^P$ with $\FF = \OO_X$.
 Then the fiber-full functor can be thought of as a refinement of the Hilbert and Quot functors due to the following inclusions.
From the tuple of functions $\mathbf{h} = (h_0,\ldots,h_r) : \ZZ^{r+1} \rightarrow \NN^{r+1}$, we define the function $P_\mathbf{h} = \sum_{i=0}^{r} (-1)^i h_i$. 
When $P_\mathbf{h} \in \QQ[m]$ is a numerical polynomial, since the Hilbert polynomial of a sheaf coincides with its Euler characteristic, we automatically get the inclusions 
$$
\fFib^{\mathbf{h}}_{X/S}(T) \; \subset \; \fHilb^{P_\mathbf{h}}_{X/S}(T) \quad \text{ and } \quad \fFib^{\mathbf{h}}_{\FF/X/S}(T) \; \subset \; \fQuot^{P_\mathbf{h}}_{\FF/X/S}(T)
$$
for any (locally Noetherian) $S$-scheme $T$.
In fact, $\fFib^{\mathbf{h}}_{\FF/X/S}$ is a locally closed subfunctor of $\fQuot^{P_\mathbf{h}}_{\FF/X/S}$ (see \autoref{rem_subfunctor}).
If $P_\mathbf{h}$ is not a numerical polynomial, then $\fFib^{\mathbf{h}}_{\FF/X/S}(T)=\emptyset$ for any $S$-scheme $T$.

\medskip

The following is the main theorem of this article.
  
\begin{headthm}[\autoref{thm_main}]
	\label{mainthmA}
	Let $S$ be a locally Noetherian scheme, $f : X \subset \PP_S^r \rightarrow S$ be a projective morphism and $\FF$ be a coherent sheaf on $X$.
	Let $\mathbf{h} = (h_0,\ldots,h_r) : \ZZ^{r+1} \rightarrow \NN^{r+1}$ be a tuple of functions and suppose that $P_\mathbf{h} \in \QQ[m]$ is a numerical polynomial.
	Then there is a quasi-projective $S$-scheme $\Fib^{\mathbf{h}}_{\FF/X/S}$ that represents the functor $\fFib^{\mathbf{h}}_{\FF/X/S}$ and that is a locally closed subscheme of the Quot scheme $\Quot^{P_\mathbf{h}}_{\FF/X/S}$.
\end{headthm}

Our primary tool for constructing the fiber-full scheme is given in \autoref{thm_flat_Sheaf} where we provide a flattening stratification theorem that deals with all the direct images of a sheaf and its possible twistings. 
To prove this technical theorem we utilize some techniques previously developed in the papers \cite{GEN_FREENESS_LOC_COHOM,FIBER_FULL}.
In a related direction, we also introduce the notion of fiber-full sheaves and we give three equivalent definitions in \autoref{thm_fib_full_sheaf}. 
Under the above notation, we say that $\FF$ is a \emph{fiber-full sheaf over $S$} if $R^i{{f}_*}\left(\FF(\nu)\right)$ is locally free over $S$ for all $0 \le i \le r$ and $\nu \in \ZZ$.
Fiber-full sheaves serve as a sheaf-theoretic extension of the notions of \emph{algebras having liftable local cohomology}  \cite{KOLLAR_KOVACS}  and \emph{cohomologically full rings} \cite{COHOM_FULL_RINGS}.

\medskip
It turns out there has been previous interest in stratifying the Hilbert scheme in terms of the whole cohomological data:
\begin{itemize}[--]
	\item In the work of Martin-Deschamps and Perrin \cite{MARTIN-DESCHAMPS_PERRIN}, they were able to control the cohomologies of a sheaf, but not all the possible twistings, as their method would yield the intersection of infinitely many (not necessarily closed) subschemes (see \cite[Chapitre VI, Proposition 1.9 and Corollaire 1.10]{MARTIN-DESCHAMPS_PERRIN}); their approach is based on classical techniques related to the Grothendieck complex which are covered, e.g., in \cite[\S III.12]{HARTSHORNE}.

	\item In the thesis of Fumasoli \cite{FUMASOLI_PAPER, FUMASOLI_THESIS}, he stratified the Hilbert scheme by bounding below the cohomological functions of the points of the Hilbert scheme, which is a consequence of the classical upper semicontinuity theorem (see \cite[Theorem III.12.8]{HARTSHORNE}).	
\end{itemize}

\noindent
Other fine moduli spaces generalizing the Hilbert scheme have been constructed by Haiman and Sturmfels \cite{HS} and by Artin and Zhang \cite{AZ}.

Our main result \autoref{mainthmA} vastly generalizes the two aforementioned approaches and shows that one can indeed stratify the Hilbert and Quot schemes by taking into account all the cohomological data.
In this regard,  one important part of our work is to develop the necessary tools that allow us to prove the general stratification result of \autoref{thm_flat_Sheaf}.

%

\smallskip

There is a large literature on the study of the loci of \emph{arithmetically Cohen-Macaulay} (ACM for short) schemes and the loci of \emph{arithmetically Gorenstein} (AG for short) schemes within the Hilbert scheme (see \cite{KLEPPE_DETERMINANTAL, KLEPPE_MAXIMAL, KLEPPE_MIRO_GOR, HARTSHORNE_SABADINI_SCHLESINGER, MARTIN-DESCHAMPS_PERRIN, ELLINGSRUD} and the references therein). 
As a result of considering the fiber-full scheme, we can provide a finer description of these loci and parametrize ACM and AG schemes with a fixed cohomological data.
Let $0 \le d \le r$ and $h_0, h_d : \ZZ \rightarrow \NN$ be two functions, and consider the tuple of functions $\mathbf{h} : \ZZ^{r+1} \rightarrow \NN^{r+1}$ given by $\mathbf{h} = (h_0, 0,\ldots,0,h_d,0,\ldots,0)$ where $0 : \ZZ \rightarrow \NN$ denotes the zero function.
To study ACM and AG schemes, since all the intermediate cohomologies vanish in these cases, it becomes natural to consider the following two functors. 
For any (locally Noetherian) $S$-scheme $T$, we have 
$$
\fACM_{X/S}^{h_0,h_d}(T) \;= \;  \left\lbrace \text{closed subscheme } Z \subset X_T \;
\begin{array}{|l}
	Z \in \fFib^{\mathbf{h}}_{X/S}(T) \text{ and $Z_t$ is  ACM for all $t\in T$} 
\end{array}
\right\rbrace
$$
and
$$
\fAG_{X/S}^{h_0,h_d}(T) \;= \;  \left\lbrace \text{closed subscheme } Z \subset X_T \;
\begin{array}{|l}
	Z \in \fFib^{\mathbf{h}}_{X/S}(T) \text{ and $Z_t$ is  AG for all $t\in T$} 
\end{array}
\right\rbrace.
$$

The following theorem shows the two functors above are representable, and so it provides the natural parameter spaces for ACM and AG schemes with fixed cohomological data.

\begin{headthm}[\autoref{thm_ACM}]
	Let $S$ be a locally Noetherian scheme and $f : X \subset \PP_S^r \rightarrow S$ be a projective morphism.
	Let $0 \le d \le r$ and $h_0, h_d : \ZZ \rightarrow \NN$ be two functions, and consider the tuple of functions $\mathbf{h}  = (h_0, 0,\ldots,0,h_d,0,\ldots,0): \ZZ^{r+1} \rightarrow \NN^{r+1}$.
	Suppose that $P_\mathbf{h} \in \QQ[m]$ is a numerical polynomial.
	Then there exist open $S$-subschemes $\ACM_{X/S}^{h_0,h_d}$ and $\AG_{X/S}^{h_0,h_d}$ of $\Fib^{\mathbf{h}}_{X/S}$ that represent the functors  $\fACM_{X/S}^{h_0,h_d}$ and $\fAG_{X/S}^{h_0,h_d}$, respectively.
\end{headthm}

\medskip

We study several examples of Hilbert schemes that we stratify in terms of fiber-full schemes.
The list includes the Hilbert scheme of points and classical examples such as the Hilbert scheme of twisted cubics and the Hilbert scheme of skew lines (see \autoref{section_examples_1}).
Furthermore, by using the recent classification of Skjelnes and Smith \cite{SMOOTH_HILB}, we show in \autoref{prop_smooth_Hilb} that smooth Hilbert schemes coincide with a fiber-full scheme (i.e., cohomological data is constant for points in a smooth Hilbert scheme).

\medskip

On top of being a generalization of the Hilbert scheme, the fiber-full scheme is also the natural parameter space for \textit{square-free} Gr\"obner degenerations, an important class of degenerations in algebraic geometry and commutative algebra.

To be more explicit, let $\kk$ be a field, $S = \kk[x_0,\ldots,x_r]$ be a polynomial ring and $>$ be a monomial order on $S$.
Given a closed subscheme $V(I) \subseteq \PP^r_\kk$, a Gr\"obner degeneration is a flat family over $\Spec(\kk[t])$ with general fiber isomorphic to $\Proj(S/I)$ and special fiber isomorphic to $\Proj(S/\text{in}_>(I))$. In particular, the natural parameter space to study Gr\"obner degenerations is the Hilbert scheme of projective space, and this point of view has been very fruitful in the study of these Hilbert schemes. For example, one proves the connectedness of Hilbert schemes, and more generally Quot schemes, by proving that any two points are connected by a chain of rational curves, each of which is a Gr\"obner degeneration \cite{HARTSHORNE_CONNECTEDNESS, PARDUE}.
Square-free Gr\"obner degenerations, those degenerations where $\text{in}_>(I)$ is square-free, have always been an important subclass of Gr\"obner degenerations, especially since they are amenable to combinatorial techniques (Stanley-Reisner theory). 
By the recent work of Conca and Varbaro \cite{CONCA_VARBARO} the cohomological data is constant along the fibers of a square-free Gr\"obner degeneration. 
Thus, the fiber-full scheme is the natural parameter space to study square-free Gr\"obner degenerations, as evidenced by \autoref{cor_sqr_free_curve}.
To phrase this in another way, when the special fiber of a Gr\"obner degeneration is square-free, the whole degeneration lies inside the same fiber-full scheme.

\medskip

Our construction of the fiber-full scheme generates many interesting questions on the structure and geometry of these schemes. 
One that is particularly interesting, and which we plan to address in a forthcoming paper, is to construct a \textit{reasonable} modular compactification of the fiber-full scheme. 
Unlike the Hilbert scheme, fixing the cohomological data, in many cases, ensures that the fiber-full scheme does not contain extraneous components. 
Ideally, one would desire it to be the closure of the fiber-full scheme inside the Hilbert scheme. 
Not only would this have quite a few enumerative applications, it would also provide a modular compactification for various interesting loci such as arithmetically Cohen-Macaulay schemes (cf.~\cite{CM_TWISTED_CUBIC , HONSEN_THESIS}), arithmetically Gorenstein schemes, twisted quartic curves and so on.

\medskip

In a subsequent paper, we did a local study of the fiber-full scheme, which includes determining a tangent obstruction theory.
As a consequence, we found a new sufficient criterion for a flat degeneration to preserve cohomology. 
For more details, see \cite{LOCAL_FIB_FULL}.

\medskip


\medskip

\noindent{\bf Organization.}
In \autoref{section_strat_module} and \autoref{section_strat_sheaf}, we provide flattening stratifications for various modules and complexes; the most important being the local cohomology modules (\autoref{thm_flat_Loc}) and the higher direct image sheaves (\autoref{thm_flat_Sheaf}). 
In \autoref{section_fiber_full_sheaf}, we develop the notion of a fiber-full sheaf and in \autoref{section_fiber_full_scheme} we construct the fiber-full scheme (\autoref{thm_main}). 
In \autoref{section_examples_1} and \autoref{section_examples_2}, we provide several examples of fiber-full schemes. 
Finally, in \autoref{section_grobner}, we consider square-free Gr\"obner degenerations and we relate them to the fiber-full scheme.

\section{Some flattening stratification theorems in a graded category of modules} 	 \label{section_strat_module}
	
In this section, we provide  flattening stratification theorems for graded modules, cohomology of complexes of modules, Ext modules, and local cohomology modules.

\subsection{Flattening stratification of modules}	

In this subsection, we concentrate on an extension for modules of the  flattening stratification theorem given in \cite{HOM_HILB_SCH} (also, see \cite[\S 8]{MUMFORD_LECTURES}).
Throughout this subsection, we shall use the following setup. 

\begin{setup}
	\label{setup_graded_flattenings}
Let $A$ be ring (always assumed to be commutative and unitary) and $R$ be a finitely generated graded $A$-algebra.
For any $\pp \in \Spec(A)$, let $\kappa(\pp) := A_\pp/\pp A_\pp$ be the residue field of $\pp$.
\end{setup}

For a graded $R$-module $M$, we say that $M$ has a \emph{Hilbert function over $A$} if for all $\nu \in \ZZ$ the graded part $[M]_\nu$ is a finitely generated locally free $A$-module of constant rank on $\Spec(A)$; and in this case, the Hilbert function is $h_M : \ZZ \rightarrow \NN, \; h_M(\nu) = \rank_A\left([M]_\nu\right)$.
If a graded $R$-module $M$ has a Hilbert function over $A$, then $M \otimes_{A} B$ has the same Hilbert function over any $A$-algebra $B$.

\begin{remark}
	\label{rem_Hilb_funct_short_ex_seq}
	Let $0 \rightarrow L \rightarrow M \rightarrow N \rightarrow 0$ be a short exact sequence of graded $R$-modules.
	\begin{enumerate}[(i)]
		\item If $L$ and $N$ have Hilbert functions over $A$, then $M$ has a Hilbert function over $A$ given by $h_M(\nu) = h_L(\nu) + h_N(\nu)$.
		\item If $M$ and $N$ have Hilbert functions over $A$, then $L$ has a Hilbert function over $A$ given by $h_L(\nu) = h_M(\nu) - h_N(\nu)$.
	\end{enumerate}
	See \cite[\S 3.1]{ROTMAN}.
\end{remark}

For completeness, we recall the following flatness result.

\begin{lemma}
	\label{lem_locus_flat}
	Assume that the ring $A$ is Noetherian.
	For any finitely
	generated graded $R$-module $M$, the locus
	$$
	U_M := \big\lbrace \pp \in \Spec(A) \mid M \otimes_{A} A_\pp \text{ is a flat $A_\pp$-module}  \big\rbrace
	$$
	is an open subset of $\Spec(A)$.
\end{lemma}
\begin{proof}
	For a proof, see \cite[Lemma 2.1]{HOM_HILB_SCH} or \cite[Lemma 2.5]{GEN_FREENESS_LOC_COHOM}.
\end{proof}

For a given graded $R$-module $M$ and a function $h : \ZZ \rightarrow \NN$, we consider the following functor from $A$-algebras to sets.  For any (Noetherian) $A$-algebra $B$,
$$
\fF_M^h(B) \;:=\;  \left\lbrace \text{morphism } \Spec(B) \rightarrow \Spec(A) \;
\begin{array}{|l}
	\left[M \otimes_{A} B\right]_\nu \text{ is a locally free $B$-module}\\
	\text{of rank $h(\nu)$ for all $\nu \in \ZZ$}
\end{array}
\right\rbrace.
$$
We  now describe our first flattening stratification theorem. 	
	
\begin{theorem}
	\label{thm_stratify_mod}
	Assume that $A$ is Noetherian.
	Let $M$ be a finitely generated graded $R$-module and $h : \ZZ \rightarrow \NN$ be a function.
	Then the following statements hold:
	\begin{enumerate}[\rm (i)]
		\item The functor $\fF_M^h$ is  represented by a locally closed subscheme $F_M^h \subset \Spec(A)$.
		In other words, for any morphism $g: \Spec(B) \rightarrow \Spec(A)$, $M \otimes_{A} B$ has a Hilbert function over $B$ equal to $h$ if and only if $g$ can be factored as 
		$$
		\Spec(B)  \; \rightarrow \; F_M^h \; \rightarrow \; \Spec(A).
		$$
		\item There is only a finite number of different functions $h_1,\ldots,h_m : \ZZ \rightarrow \NN$ such that $F_M^{h_i} \neq \emptyset$, and so $\Spec(A)$ is set-theoretically equal to the disjoint union of the locally closed subschemes $F_{M}^{h_i}$.
	\end{enumerate}
\end{theorem}	
\begin{proof}
	(i) For any morphism $\Spec(B) \rightarrow \Spec(A)$, one has that $\left[M \otimes_{A} B\right]_\nu$ is locally free of rank $h(\nu)$ if and only if $\Fitt_{h(\nu)-1}([M \otimes_{A} B]_\nu) = 0$ and $\Fitt_{h(\nu)}([M \otimes_{A} B]_\nu) = B$, and that $\Fitt_{j}([M \otimes_{A} B]_\nu) = \left(\Fitt_{j}([M]_\nu)\right) B$ (for more details on Fitting ideals, see \cite[\S 20.2]{EISEN_COMM}).
	
		Let $Z_M^h \subset \Spec(A)$ be the closed subscheme given by $Z_M^h  := \Spec\left(A / \left( \sum_{\nu \in \ZZ} \Fitt_{h(\nu)-1}([M]_\nu)\right)\right)$.
	We have that $\Fitt_{h(\nu)-1}([M \otimes_{A} B]_\nu) =0$ for all $\nu \in \ZZ$ if and only if $\Spec(B) \rightarrow \Spec(A)$ factors through $Z_M^h$.
	Therefore, we can substitute $\Spec(A)$ by $Z_M^h$, and we do so.
	
	Let $\pp \in U_M$ and suppose that $M \otimes_A A_\pp$ has a Hilbert function $h_{M \otimes_A A_\pp} = h$ over $A_\pp$.
	By \autoref{lem_locus_flat}, there is an affine open neighborhood $\pp \in \Spec(A_a) \subset U_M$ of $\pp$ for some $a \in A$.
	Thus \cite[\href{https://stacks.math.columbia.edu/tag/00NX}{Tag 00NX}]{stacks-project} implies that for all $\nu \in \ZZ$ the function $\Spec(A_a) \rightarrow \NN, \; \qqq \mapsto \dim_{\kappa(\qqq)}([M \otimes_{A_a} \kappa(\qqq)]_\nu)$ is locally constant.
	Consequently, there is an open connected neighborhood $\pp \in V \subset \Spec(A_a)$  of $\pp$ such that $h_{M \otimes_A A_\qqq} = h$ for all $\qqq \in V$.
	It then follows that the following locus
	$$
	U_M^h := \big\lbrace \pp \in \Spec(A) \mid M \otimes_{A} A_\pp \text{ has a Hilbert function $h_{M \otimes_{A} A_\pp} = h$ over $A_\pp$}  \big\rbrace
	$$
	is an open subset of $\Spec(A)$.

	Note that $\Fitt_{h(\nu)}([M \otimes_{A} B]_\nu) = B$ if and only if $\mathfrak{P} \cap A \not\supset \Fitt_{h(\nu)}([M]_\nu)$ for all $\mathfrak{P} \in \Spec(B)$.
	Hence, under the condition $\Fitt_{h(\nu)-1}([M \otimes_{A} B]_\nu) =0$ for all $\nu \in \ZZ$, it follows that $\Fitt_{h(\nu)}([M \otimes_{A} B]_\nu) = B$ for all $\nu \in \ZZ$ if and only if  $\Spec(B) \rightarrow \Spec(A)$ factors through $U_M^h$.
	So, after having changed $\Spec(A)$ by $Z_M^h$, we have that  $\fF_M^h$ is represented by the open  subscheme $U_M^h \subset \Spec(A)$.
	
	(ii) For each $\pp \in \Spec(A)$, let $h_\pp$ be the function $h_\pp(\nu) := \dim_{\kappa(\pp)}([M \otimes_{A} \kappa(\pp)]_\nu)$.
	As we have a natural morphism $\Spec(\kappa(\pp)) \rightarrow \Spec(A)$, it clearly follows that $\pp \in F^{h_\pp}_M$.
	Therefore, by the Noetherian hypothesis,  there is a finite number of distinct functions $h_1,\ldots,h_m$ such that set-theoretically we have the equality $\Spec(A) = \bigsqcup_{i=1}^m F_M^{h_i}$.
\end{proof}

\subsection{Flattening stratification of the cohomologies of a complex}

Here we study how the process of taking tensor product with another ring affects the cohomology of a bounded complex.
In this subsection, we continue using \autoref{setup_graded_flattenings}.
The notation below is used throughout the paper.

\begin{notation}
	For a (co-)complex of $A$-modules
	$
	K^\bullet: \, \cdots \rightarrow K^{i-1} \xrightarrow{\phi^{i-1}} K^i \xrightarrow{\phi^i}   K^{i+1} \rightarrow \cdots,
	$
	one defines
	$\zZZ^i\left(K^\bullet\right) := \Ker(\phi^i)$,	$\bBB^i\left(K^\bullet\right) := \IM(\phi^{i-1})$,
	$\HH^i\left(K^\bullet\right) := \zZZ^i(K^\bullet)/\bBB^i(K^\bullet)$,	and $\cCC^i\left(K^\bullet\right) := K^i/\bBB^i(K^\bullet) \,\supset\, \HH^i\left(K^\bullet\right)$
	for all $i \in \ZZ$.
	We use analogous notation with lower indices for a complex $K_\bullet$.
\end{notation}

\begin{remark}
	\label{rem_eq_dim_in_terms_of_Cokers}
	For a complex of $A$-modules $K^\bullet$ and an $A$-module $N$, we have a four-term exact sequence
	\begin{align*}	
		0 \rightarrow \HH^i\big(K^\bullet \otimes_A N\big) \rightarrow \cCC^i(K^\bullet) \otimes_A N \rightarrow 
		K^{i+1} \otimes_A N \rightarrow \cCC^{i+1}(K^\bullet) \otimes_A N \rightarrow 0
	\end{align*}
	of $A$-modules.
	See, e.g., \cite[Proof of Theorem 12.8]{HARTSHORNE}.
\end{remark}

The following lemma transfers the burden of studying the cohomologies of a bounded complex to considering the cokernels of the maps.

\begin{lemma}
	\label{lem_HH_coker_relations}
	Let $K^\bullet : 0 \rightarrow K^0 \rightarrow K^1 \rightarrow \cdots \rightarrow K^p \rightarrow 0$ be a bounded complex of graded $R$-modules.
	Suppose that each $K^i$ has a Hilbert function over $A$.
	Let $\Spec(B) \rightarrow \Spec(A)$ be a morphism.
	Then the following two conditions are equivalent: 
	\begin{enumerate}[\rm (1)]
		\item $\HH^i(K^\bullet \otimes_{A} B)$ has a Hilbert function over $B$ for all $0 \le i \le p$.
		\item $\cCC^i(K^\bullet) \otimes_{A} B$ has a Hilbert function over $B$ for all $0 \le i \le p$.
	\end{enumerate}
	Moreover, if any of the above conditions are satisfied, we have 
	$$
	h_{\HH^i(K^\bullet \otimes_{A} B)} \;=\; h_{\cCC^i(K^\bullet) \otimes_{A} B} + h_{\cCC^{i+1}(K^\bullet) \otimes_{A} B} - h_{K^{i+1} \otimes_{A} B}
	$$
	 and 
	 $$
	 h_{\cCC^i(K^\bullet) \otimes_{A} B}  \;=\; \sum_{j = i}^{p} (-1)^{j-i} \left( h_{\HH^j(K^\bullet \otimes_{A} B)} + h_{K^{j+1} \otimes_{A} B}  \right).
	$$
\end{lemma}
\begin{proof}
	We have the four-term exact sequence $0 \rightarrow \HH^i\big(K^\bullet \otimes_A B\big) \rightarrow \cCC^i(K^\bullet) \otimes_A B \rightarrow 
	K^{i+1} \otimes_A B \rightarrow \cCC^{i+1}(K^\bullet) \otimes_A B \rightarrow 0$, which can be broken into short exact sequences 
	$$
	0 \rightarrow \HH^i\big(K^\bullet \otimes_A B\big) \rightarrow  \cCC^i(K^\bullet) \otimes_A B \rightarrow  L^i \rightarrow 0 \quad \text{ and } \quad 	0 \rightarrow L^i \rightarrow  	K^{i+1} \otimes_A B \rightarrow \cCC^{i+1}(K^\bullet) \otimes_A B \rightarrow 0
	$$
	where $L^i$ is some graded $R$-module.
	
	By \autoref{rem_Hilb_funct_short_ex_seq}, if all $\cCC^i(K^\bullet) \otimes_{A} B$ have a Hilbert function over $B$ then all $L^i$ have a Hilbert function over $B$ and, by the same token, it follows that all $\HH^i(K^\bullet \otimes_A B)$ have a Hilbert function over $B$.
	This establishes the implication (2) $\Rightarrow$ (1).
	
	Suppose that all $\HH^i(K^\bullet \otimes_A B)$ have a Hilbert function over $B$.
	As a consequence of \autoref{rem_Hilb_funct_short_ex_seq}, if $\cCC^{i+1}(K^\bullet) \otimes_A B$ has a Hilbert function over $B$, we obtain that $L^i$ and, subsequently, $\cCC^{i}(K^\bullet) \otimes_A B$ have Hilbert functions over $B$.
	Since $\cCC^{p}(K^\bullet) \otimes_A B = \HH^p(K^\bullet \otimes_A B)$, by descending induction on $i$, we get that all $\cCC^i(K^\bullet) \otimes_A B$ have a Hilbert function over $B$.
	So, the other implication (1) $\Rightarrow$ (2) also holds.
	
	The additional equations relating the Hilbert functions of $\cCC^{i}(K^\bullet) \otimes_A B$ and $\HH^i(K^\bullet \otimes_A B)$ follow similarly.
\end{proof}

For a given bounded complex of graded $R$-modules $K^\bullet: 0 \rightarrow K^0 \rightarrow K^1 \rightarrow \cdots \rightarrow K^p \rightarrow 0$ such that each $K^i$ has a Hilbert function over $A$ and a given tuple of $p+1$ functions $\mathbf{h} = (h_0,\ldots,h_p) : \ZZ^{p+1} \rightarrow \NN^{p+1}$, we consider the following functor for any $A$-algebra $B$,
$$
\fF_{K^\bullet}^\mathbf{h}(B) \;:=\;  \left\lbrace \text{morphism } \Spec(B) \rightarrow \Spec(A) \;
\begin{array}{|l}
	\left[\HH^i(K^\bullet \otimes_{A} B)\right]_\nu \text{ is a locally free $B$-module}\\
	\text{of rank $h_i(\nu)$ for all $0 \le i \le p, \nu \in \ZZ$}
\end{array}
\right\rbrace.
$$
This is a functor from (Noetherian) $A$-algebras to sets.
For completeness, we include a lemma which shows that, in our setting, flatness is equivalent to being locally free.

\begin{lemma}
	\label{lem_equiv_flat_loc_free}
	Let $K^\bullet : 0 \rightarrow K^0 \rightarrow K^1 \rightarrow \cdots \rightarrow K^p \rightarrow 0$ be a bounded complex of graded $R$-modules.
	Suppose that $[K^i]_\nu$ is a finitely generated locally free $A$-module for all $0 \le i \le p, \nu \in \ZZ$.
	Let $\Spec(B) \rightarrow \Spec(A)$ be a morphism.
	Then the following two conditions are equivalent: 
	\begin{enumerate}[\rm (1)]
		\item $[\HH^i(K^\bullet \otimes_{A} B)]_\nu$ is a flat $B$-module for all $0 \le i \le p, \nu \in \ZZ$.
		\item $[\HH^i(K^\bullet \otimes_{A} B)]_\nu$ is a locally free $B$-module for all $0 \le i \le p, \nu \in \ZZ$.
	\end{enumerate}
\end{lemma}
\begin{proof}	
	The implication (2) $\Rightarrow$ (1) is clear. 
	So, we assume that each $[\HH^i(K^\bullet \otimes_{A} B)]_\nu$ is a flat $B$-module.
	As in the proof of \autoref{lem_HH_coker_relations}, we consider the short exact sequences $0 \rightarrow \HH^i\big(K^\bullet \otimes_A B\big) \rightarrow  \cCC^i(K^\bullet) \otimes_A B \rightarrow  L^i \rightarrow 0$ and $0 \rightarrow L^i \rightarrow  	K^{i+1} \otimes_A B \rightarrow \cCC^{i+1}(K^\bullet) \otimes_A B \rightarrow 0$.
	Each $[\cCC^{p}(K^\bullet) \otimes_A B]_\nu = [\HH^p(K^\bullet \otimes_A B)]_\nu$ is a locally free $B$-module since it is flat of finite presentation as a $B$-module.
	Similarly to \autoref{lem_HH_coker_relations}, by descending induction on $i$, we can show that $[\HH^i(K^\bullet \otimes_{A} B)]_\nu$ and $[\cCC^i(K^\bullet) \otimes_{A} B]_\nu$ are locally free $B$-modules for all $0 \le i \le p, \nu \in \ZZ$. 
\end{proof}

The following theorem deals with the stratification of the cohomologies of bounded complexes.

\begin{theorem}
		\label{thm_flattening_complex}
		Assume that $A$ is Noetherian.
		Let $K^\bullet : 0 \rightarrow K^0 \rightarrow K^1 \rightarrow \cdots \rightarrow K^p \rightarrow 0$ be a bounded complex of finitely generated graded $R$-modules and $\mathbf{h} = (h_0,\ldots,h_p) : \ZZ^{p+1} \rightarrow \NN^{p+1}$ be a tuple of functions.
		Suppose that each $K^i$ has a Hilbert function over $A$.
		Then the functor $\fF_{K^\bullet}^\mathbf{h}$ is  represented by a locally closed subscheme $F_{K^\bullet}^\mathbf{h} \subset \Spec(A)$.
		In other words, for any morphism $g: \Spec(B) \rightarrow \Spec(A)$, each $\HH^i(K^\bullet \otimes_{A} B)$ has a Hilbert function over $B$ equal to $h_i$ if and only if $g$ can be factored as 
		$$
		\Spec(B)  \; \rightarrow \; F_{K^\bullet}^\mathbf{h} \; \rightarrow \; \Spec(A).
		$$
\end{theorem}
\begin{proof}
	For any morphism $\Spec(B) \rightarrow \Spec(A)$, \autoref{lem_HH_coker_relations} implies that each $\HH^i(K^\bullet \otimes_{A} B)$ has a Hilbert function over $B$ equal to $h_i$ if and only if each $\cCC^i(K^\bullet) \otimes_{A} B$ has a Hilbert function over $B$ equal to $h_i'$, where $h_i' := \sum_{j = i}^{p} (-1)^{j-i} \left( h_j + h_{K^{j+1} \otimes_{A} B}  \right)$.
	Therefore, by \autoref{thm_stratify_mod}, $\fF_{K^\bullet}^\mathbf{h}$ is  represented by the locally closed subscheme $F_{K^\bullet}^\mathbf{h} \subset \Spec(A)$ given by $F_{\cCC^0(K^\bullet)}^{h_0'} \cap F_{\cCC^1(K^\bullet)}^{h_1'} \cap \cdots \cap F_{\cCC^p(K^\bullet)}^{h_p'}$.
\end{proof}

\subsection{Flattening stratification of Ext modules}
We now focus on a flattening stratification result for certain Ext modules.
During this subsection, we shall use the following setup. 

\begin{setup}
	\label{setup_flattening_poly_ring}
	Let $A$ be a Noetherian ring and $R$ be a positively graded polynomial ring $R = A[x_1,\ldots,x_r]$ over $A$; that is, $\deg(x_i) \in \ZZ_+$ and $\deg(a)=0 \in \NN$ for all $a \in A$.
\end{setup}

First, we recall the following result from \cite{FIBER_FULL}.

\begin{lemma}
	\label{lem_Ext_base_change}
	Let $M$ be a finitely generated graded $R$-module and suppose that $M$ is a flat $A$-module.
	Let $F_\bullet :  \cdots \rightarrow F_i \rightarrow \cdots \rightarrow  F_1 \rightarrow F_0$ be a graded free $R$-resolution of $M$ by modules of finite rank.
	Let 
	$$
	D_M^i := \Coker\big(\Hom_R(F_{i-1}, R) \rightarrow \Hom_R(F_{i}, R)\big)
	$$
	for each $i \ge 0$.
	Then the following statements hold:
	\begin{enumerate}[\rm (i)]
		\item $\Ext_R^i(M, R) = 0$ for all $i \ge r+1$.
		\item $D_M^i$ is a flat $A$-module for all $i \ge r+1$.
		\item If $\Ext_R^i(M, R)$ is a flat $A$-module for all $0 \le i \le r$, then $$
		\Ext_R^i(M, R) \otimes_A B \xrightarrow{\cong} \Ext_{R \otimes_A B}^i(M\otimes_A B, R \otimes_A B)
		$$
		for all $i \ge 0$ any $A$-algebra $B$.
	\end{enumerate}
\end{lemma}
\begin{proof}
	It follows directly from \cite[Lemma 2.10]{FIBER_FULL}.
\end{proof}

For a given finitely generated graded $R$-module $M$ that is $A$-flat and a tuple of functions $\mathbf{h} = (h_0,\ldots,h_r) : \ZZ^{r+1} \rightarrow \NN^{r+1}$, we consider the following functor for any $A$-algebra $B$, 
$$
\fFExt_M^{\mathbf{h}}(B) \;:=\;  \left\lbrace \text{morphism } \Spec(B) \rightarrow \Spec(A) \;
\begin{array}{|l}
	\left[\Ext_{R \otimes_A B}^i(M \otimes_{A} B, R\otimes_{A} B)\right]_\nu \text{ is a locally free} \\ 
	\text{$B$-module of rank $h_i(\nu)$ for all $0 \le i \le r, \nu \in \ZZ$}
\end{array}
\right\rbrace.
$$
This is a functor from (Noetherian) $A$-algebras to sets.
This functor controls all the Ext modules of $M$ because, as a consequence of \autoref{lem_Ext_base_change}, if $M$ is $A$-flat then $\Ext_{R \otimes_A B}^i(M \otimes_{A} B, R \otimes_{A} B)=0$ for all $i \ge r+1$.
The next theorem provides a flattening stratification for all the Ext modules.

\begin{theorem}
	\label{thm_flat_Ext}
	Let $M$ be a finitely generated graded $R$-module that is a flat $A$-module, and  $\mathbf{h} = (h_0,\ldots,h_r) : \ZZ^{r+1} \rightarrow \NN^{r+1}$ be a tuple of functions.
	Then the functor $\fFExt_{M}^{\mathbf{h}}$ is  represented by a locally closed subscheme $\FExt_{M}^{\mathbf{h}} \subset \Spec(A)$.
	In other words, for any morphism $g: \Spec(B) \rightarrow \Spec(A)$, each $\Ext_{R \otimes_A B}^i(M \otimes_{A} B, R \otimes_{A} B)$ has a Hilbert function over $B$ equal to $h_i$ if and only if $g$ can be factored as 
	$$
		\Spec(B)  \; \rightarrow \; \FExt_{M}^{\mathbf{h}} \; \rightarrow \; \Spec(A).
	$$
\end{theorem}
\begin{proof}
	Let $F_\bullet: \cdots \rightarrow F_i \rightarrow \cdots \rightarrow  F_1 \rightarrow F_0$ be a graded free $R$-resolution of $M$ by modules of finite rank.
	Consider the complex $F_\bullet^{\le r+1}$ given as the truncation
	$
	F_\bullet^{\le r+1}  :  0 \rightarrow F_{r+1} \rightarrow F_r \rightarrow \cdots \rightarrow F_1 \rightarrow F_0, 
	$
	and $P^\bullet := \Hom_R(F_\bullet^{\le r+1}, R)$.
	By \autoref{lem_Ext_base_change}, $D_M^{r+1} = \HH^{r+1}(P^\bullet) = \cCC^{r+1}(P^\bullet)$ is a flat $A$-module and so each $[D_M^{r+1}]_\nu$ (being finitely presented over $A$) is a locally free $A$-module.
	Hence \cite[\href{https://stacks.math.columbia.edu/tag/00NX}{Tag 00NX}]{stacks-project} implies that for all $\nu \in \ZZ$ the function $\Spec(A) \rightarrow \NN, \; \pp \mapsto \dim_{\kappa(\pp)}([D_M^{r+1} \otimes_{A} \kappa(\pp)]_\nu)$ is locally constant.
	As a consequence, $h_{D_M^{r+1}}$ is a constant function on each connected component of $\Spec(A)$.
	
	Consider the bounded complex  $K^\bullet$ given by 
	$$
	K^\bullet : \quad 0 \rightarrow P^0 \rightarrow \cdots \rightarrow P^r \rightarrow P^{r+1} \rightarrow D_M^{r+1} \rightarrow 0.
	$$ 	
	Note that $\HH^i(K^\bullet \otimes_{A} B) = \HH^i(P^\bullet \otimes_{A} B) \cong \Ext_{R \otimes_A B}^i(M \otimes_{A} B, R\otimes_{A} B)$ for all $0 \le i \le r$ (since $M$ is $A$-flat), and that $\HH^{r+1}(K^\bullet \otimes_{A} B) = \HH^{r+2}(K^\bullet \otimes_{A} B) = 0$.
	
	To show that the functor $\fFExt_M^{\mathbf{h}}$ is representable, we can simply restrict $\Spec(A)$ to one of its connected components.
	Thus, we now assume that $\Spec(A)$ is connected, and so $D_M^{r+1}$ has a Hilbert function over $A$.
 	Let $\mathbf{h}' = (h_0,\ldots,h_r,0,0) : \ZZ^{r+3} \rightarrow \NN^{r+3}$ be obtained by concatenating two zero functions $0 : \ZZ \rightarrow \NN$ to $\mathbf{h}$.
 	Finally, by \autoref{thm_flattening_complex}, it follows that $\fFExt_M^{\mathbf{h}}$ is represented by the locally closed subscheme $\FExt_M^{\mathbf{h}} := F_{K^\bullet}^{\mathbf{h}'} \subset \Spec(A)$.
 	This settles the proof of the theorem.
\end{proof}

\subsection{Flattening stratification of local cohomology modules}
Next, we provide a flattening stratification theorem for local cohomology modules.
The main idea is that, by using some techniques from \cite{FIBER_FULL, GEN_FREENESS_LOC_COHOM}, we can obtain a flattening stratification of local cohomology modules from the one of Ext modules given in \autoref{thm_flat_Ext}.

We start with the following lemma that gives a base change of local cohomology modules over a base which is not necessarily Noetherian. 

\begin{lemma}
	\label{lem_base_change_loc}
	Let $A$ be a ring, $R = A[x_1,\ldots,x_r]$ be a positively graded polynomial ring over $A$, $\mm = (x_1,\ldots,x_r) \subset R$ be the graded irrelevant ideal, and $M$ be a graded $R$-module.
	If $M$ is  $A$-flat and $\HL^i(M)$ is  $A$-flat for all $0 \le i \le r$, then $\HL^i(M) \otimes_{A} B \xrightarrow{\cong} \HL^i(M \otimes_{A} B)$ for all $0 \le i \le r$ and any $A$-algebra $B$. 
\end{lemma}
\begin{proof}
	By using \cite{SCHENZEL_PROREGULAR} and the fact that $x_1,\ldots,x_r$ is a regular sequence in $R$, even if $A$ is not Noetherian, we can compute $\HL^i(M)$ as the $i$-th cohomology of $\mathcal{C}_\mm^\bullet \otimes_R M$ where $\mathcal{C}_\mm^\bullet$ denotes the \v{C}ech complex with respect to $\mm = (x_1,\ldots,x_r)$.
	Let $L_\bullet  :  \cdots \rightarrow L_i \rightarrow \cdots \rightarrow  L_1 \rightarrow L_0$ be a graded free $R$-resolution of $M$.
	By considering the spectral sequences coming from the double complex $\mathcal{C}_\mm^\bullet \otimes_S L_\bullet \otimes_{A} B$, we obtain the isomorphisms 
	$$
	\HL^i(M \otimes_{A} B) \cong \HH_{r-i}\big(\HL^{r}(L_\bullet) \otimes_{A} B\big)
	$$
	for any $A$-algebra $B$ and all integers $i$ (see \cite[Lemma 3.4]{GEN_FREENESS_LOC_COHOM}).
	By the flatness condition and standard base change results (see \cite[Lemma 2.8]{GEN_FREENESS_LOC_COHOM}), we obtain
	$$
	\HL^i(M)\otimes_{A} B \cong \HH_{r-i}(\HL^{r}(L_\bullet)) \otimes_{A} B \xrightarrow{\cong} \HH_{r-i}(\HL^{r}(L_\bullet) \otimes_{A} B) \cong  \HL^i(M \otimes_{A} B),
	$$
	and so the result follows.
\end{proof}

The following setup is now set in place for the rest of the subsection.

\begin{setup}
	\label{setup_flattening_cohom}
	Let $A$ be a Noetherian ring, $R$ be a positively graded polynomial ring $R = A[x_1,\ldots,x_r]$ over $A$, $\mm = (x_1,\ldots,x_r) \subset R$ be the graded irrelevant ideal, and $\delta:=\deg(x_1)+\cdots+\deg(x_r) \in \ZZ_+$.
\end{setup}

For a graded $R$-module $M$ and a morphism $\Spec(B) \rightarrow \Spec(A)$, we consider the graded $(R \otimes_{A} B)$-module $M \otimes_{A} B$ and we denote the \emph{$B$-relative graded Matlis dual} by
$$
	\left(M \otimes_{A} B\right)^{*_B}={}^*\Hom_B(M \otimes_{A} B, B): = \bigoplus_{\nu \in \ZZ} \Hom_B\left({\left[M \otimes_{A} B\right]}_{-\nu}, B\right).
$$ 
The dual $\left(M \otimes_{A} B\right)^{*_B}$ has a natural structure of graded $(R \otimes_{A} B)$-module.
From the canonical perfect pairing of free $A$-modules in ``top'' local cohomology
$
{\left[R\right]}_\nu \otimes_{A} {\left[\HL^{r}(R)\right]}_{-\delta-\nu} \rightarrow {\left[\HL^{r}(R)\right]}_{-\delta} \cong A
$
we obtain a canonical graded $R$-isomorphism 
$
\HL^{r}(R) \cong {\left(R(-\delta)\right)}^{*_A} = \grHom_A\left(R(-\delta), A\right).	
$
Then for a morphism $\Spec(B) \rightarrow \Spec(A)$ and a complex $F_\bullet :  \cdots \rightarrow F_i \rightarrow \cdots \rightarrow  F_1 \rightarrow F_0$ of finitely generated graded free $R$-modules, we obtain the isomorphisms of complexes 
\begin{equation*}
	\HL^{r}(F_\bullet \otimes_{A} B)  \cong
	\HL^{r}(F_\bullet) \otimes_{A} B \cong {\big(\Hom_R(F_\bullet
		, R(-\delta))\big)}^{*_A} \otimes_{A} B \cong {\big(\Hom_R(F_\bullet
		, R(-\delta)) \otimes_{A} B\big)}^{*_B}.
\end{equation*}
The next proposition gives a sort of local duality theorem (see \cite[Proposition 2.11]{FIBER_FULL}).

\begin{proposition}
	\label{prop_equiv_freeness}
	Let $M$ be a finitely generated graded $R$-module and suppose that $M$ is a flat $A$-module.
	Let $\Spec(B) \rightarrow \Spec(A)$ be a morphism.
	Then the following two conditions are equivalent:
	\begin{enumerate}[\rm (1)]
		\item $\HL^i(M \otimes_{A} B)$ has a Hilbert function over $B$ for all $0 \le i \le r$.
		\item $\Ext_{R \otimes_{A} B}^i(M \otimes_{A} B, R \otimes_{A} B)$ has a Hilbert function over $B$ for all $0 \le i \le r$.
	\end{enumerate}
	Moreover, when any of the above equivalent conditions is satisfied, we have that 
	$$
	h_{\HL^i(M \otimes_{A} B)}(\nu) \;=\; h_{\Ext_{R \otimes_{A} B}^{r-i}(M \otimes_{A} B, R \otimes_{A} B)}(-\nu - \delta) 
	$$ 
	for all $i, \nu \in \ZZ$.
\end{proposition}
\begin{proof}
		Let $F_\bullet: \cdots \rightarrow F_i \rightarrow \cdots \rightarrow  F_1 \rightarrow F_0$ be a graded free $R$-resolution of $M$ by modules of finite rank.
		As $M$ is $A$-flat, $F_\bullet \otimes_{A} B$ is a resolution of $M \otimes_{A} B$.
		Then by using the isomorphism of complexes $\HL^{r}(F_\bullet \otimes_{A} B)  \cong {\big(\Hom_R(F_\bullet
			, R(-\delta)) \otimes_{A} B\big)}^{*_B}$ and the same proof of \cite[Proposition 2.11]{FIBER_FULL}, we obtain that conditions (1) and (2) are equivalent, and that in the case they are satisfied, we have the isomorphism $\HL^i(M \otimes_{A} B) \cong \big(\Ext_{R \otimes_{A} B}^{r-i}(M \otimes_{A} B, R(-\delta) \otimes_{A} B)\big)^{*_B}$.
\end{proof}

For a given finitely generated graded $R$-module $M$ that is $A$-flat and a tuple of functions $\mathbf{h} = (h_0,\ldots,h_r) : \ZZ^{r+1} \rightarrow \NN^{r+1}$, we consider the following functor for any ring $B$, 
$$
\fFLoc_M^{\mathbf{h}}(B) \;:=\;  \left\lbrace \text{morphism } \Spec(B) \rightarrow \Spec(A) \;
\begin{array}{|l}
	\left[\HL^i(M \otimes_{A} B)\right]_\nu \text{ is a locally free $B$-module} \\ 
	\text{of rank $h_i(\nu)$ for all $0 \le i \le r, \nu \in \ZZ$}
\end{array}
\right\rbrace.
$$
Finally, we have below a theorem that gives a flattening stratification for local cohomology modules.

\begin{theorem}
	\label{thm_flat_Loc}
	Let $M$ be a finitely generated graded $R$-module that is a flat $A$-module, and  $\mathbf{h} = (h_0,\ldots,h_r) : \ZZ^{r+1} \rightarrow \NN^{r+1}$ be a tuple of functions.
	Then the functor $\fFLoc_M^{\mathbf{h}}$ is  represented by a locally closed subscheme $\FLoc_M^{\mathbf{h}} \subset \Spec(A)$.
	In other words, for any morphism $g: \Spec(B) \rightarrow \Spec(A)$, each $\HL^i(M \otimes_{A} B)$ has a Hilbert function over $B$ equal to $h_i$ if and only if $g$ can be factored as 
	$$
		\Spec(B)  \; \rightarrow \; \FLoc_{M}^{\mathbf{h}} \; \rightarrow \; \Spec(A).
	$$
\end{theorem}
\begin{proof}
	Let $\mathbf{h}' = (h_0',\ldots,h_r') : \ZZ^{r+1} \rightarrow \NN^{r+1}$ be a tuple of functions defined by $h_i'(\nu) := h_{r-i}(-\nu-\delta)$.
	So, it follows directly from \autoref{prop_equiv_freeness} and \autoref{thm_flat_Ext} that $\fFLoc_M^{\mathbf{h}}$ is represented by the locally closed subscheme $\FLoc_M^{\mathbf{h}} := \FExt_M^{\mathbf{h}'} \subset \Spec(A)$.
\end{proof}

\section{Flattening  stratification of the higher direct images of a sheaf and its twistings} \label{section_strat_sheaf}

In this section, we provide a flattening stratification theorem that deals with all the direct images of a sheaf and its possible twistings. 
This result is the core of our approach to show that the fiber-full scheme exists.

For completeness, we start with a base change result which is probably well-known to the experts, but we could not find it in the generality we need (cf.~\cite[Lemma 4.1]{HARA_CHAR}).
Let $S$ be a scheme and $f : X \subset \PP_S^r \rightarrow S$ be a projective morphism.
Let $g : T \rightarrow S$ be a morphism of schemes and $t \in T$ be a point.
For brevity, we use the notation $X_T := X \times_S T$, $f_{(T)}:= f \times_S T  : X_T \rightarrow T$, $X_t := X_T \times_T \Spec(\kappa(t))$ and $f_{(t)}:= f_{(T)} \times_T \Spec(\kappa(t))  : X_t \rightarrow \Spec(\kappa(t))$, and we consider the commutative diagram
\begin{equation*}		
	\begin{tikzpicture}[baseline=(current  bounding  box.center)]
		\matrix (m) [matrix of math nodes,row sep=3.7em,column sep=5.5em,minimum width=2em, text height=1.5ex, text depth=0.25ex]
		{
			X_t = X_T \times_T \Spec(\kappa(t)) & X_T = X \times_S T &  X = X \times_S S  \\
			\Spec(\kappa(t)) & T& S .\\
		};						
		\path[-stealth]
		(m-1-1) edge node [above] {$1 \times_T \iota_t$} (m-1-2)
		(m-2-1) edge node [above] {$\iota_t$} (m-2-2)
		(m-1-1) edge node [left] {$f_{(t)}$} (m-2-1)
		(m-1-2) edge node [above] {$1 \times_S g$} (m-1-3)
		(m-2-2) edge node [above] {$g$} (m-2-3)
		(m-1-3) edge node [right] {$f$} (m-2-3)
		(m-1-2) edge node [right] {$f_{(T)}$} (m-2-2)
		;
	\end{tikzpicture}	
\end{equation*}
For a quasi-coherent sheaf $\FF$ on $X$, let $\FF_T := (1 \times_S g)^* \FF$ be the sheaf on $X_T$ obtained by the pull-back induced by $g$ and $\FF_t := (1 \times_T \iota_t)^*\FF_T$ be the sheaf on $X_t$ obtained by taking the fiber over $t$.
In this setting, we have the base change map $g^*R^if_*\FF \rightarrow R^i{{f_{(T)}}_*}\left(\FF_T\right)$ for all $i \ge 0$.

\begin{proposition}
	\label{prop_base_change_sheaf}
	Let $S$ be a scheme, $f : X \subset \PP_S^r \rightarrow S$ be a projective morphism and $\FF$ be a quasi-coherent $\OO_X$-module.
	Suppose that $R^if_*\left(\FF(\nu)\right)$ is a flat $\OO_S$-module for all $0 \le i \le r, \nu \in \ZZ$.
	Let $g : T \rightarrow S$ be a morphism of schemes.
	Then $\FF$ is flat over $S$ and we have a base change isomorphism 
	$$
	g^*R^if_*\left(\FF(\nu)\right) \xrightarrow{\cong} R^i{{f_{(T)}}_*}\left(\FF_T(\nu)\right)
	$$ 
	for all $0 \le i \le r, \nu \in \ZZ$.
\end{proposition}
\begin{proof}
	Since the first consequence is local on $S$ and the second one is local on $T$, we may assume that $T  = \Spec(B)$ and $S = \Spec(A)$ are affine schemes.
	Then we have the identifications
	$$
	R^if_*\left(\FF(\nu)\right) \cong \HH^i(X, \FF(\nu))^\thicksim
	\cong \HH^i(\PP_A^r, \FF(\nu))^\thicksim
	$$ 
	and 
	$$  
	 R^i{{f_{(T)}}_*}\left(\FF_T(\nu)\right) \cong \HH^i(X_T, \FF_T(\nu))^\thicksim \cong \HH^i(\PP_B^r, \FF_T(\nu))^\thicksim
	$$
	(see \cite[\href{https://stacks.math.columbia.edu/tag/01XK}{Tag 01XK}]{stacks-project}, \cite[Proposition 8.5]{HARTSHORNE}).
	Let $R := A[x_0,\ldots,x_r]$ with $\PP_A^r = \Proj(R)$, $\mm = (x_0,\ldots,x_r)$, and $M$ be the graded $R$-module given by $M := \bigoplus_{\nu \in \ZZ} \HH^0(\PP_A^r, \FF(\nu))$.
	Note that $\FF \cong M^\thicksim$ and $\FF_T \cong \left(M \otimes_{A} B\right)^\thicksim$.
	Thus, it is clear that $\FF$ is flat over $S$.
	We have the exact sequence 
	$$
	0 \rightarrow \HL^0(M \otimes_{A} B) \rightarrow M \otimes_{A} B \rightarrow \bigoplus_{\nu \in \ZZ} \HH^0(\PP_B^r, \FF_T(\nu)) \rightarrow \HL^1(M \otimes_{A} B) \rightarrow 0
	$$ 
	and the isomorphism $\HL^{i+1}(M \otimes_{A} B) \cong \bigoplus_{\nu \in \ZZ} \HH^i(\PP_B^r, \FF_T(\nu))$ for all $i \ge 1$.
	In the special case $B = A$, since $M = \bigoplus_{\nu \in \ZZ} \HH^0(\PP_A^r, \FF(\nu))$, we obtain that $\HL^0(M) = \HL^1(M) = 0$.
	Finally, \autoref{lem_base_change_loc} implies that $\HL^i(M) \otimes_{A} B \xrightarrow{\cong} \HL^i(M \otimes_{A} B)$ for all $0 \le i \le r+1$, and so the proof of the proposition is complete.
\end{proof}

We fix the following setup for the rest of this section.
\begin{setup}
	\label{setup_sheaf_flat}
	Let $S$ be a locally Noetherian scheme and $f : X \subset \PP_S^r \rightarrow S$ be a projective morphism.
\end{setup}

When we take the fiber $X_t = X_T \times_T \Spec(\kappa(t))$ of $f_{(T)}$ over $t \in T$, we get the isomorphism 
$$
R^i{{f_{(t)}}_*}\left(\FF_t\right) \; \cong \; \HH^i\left(X_t,\FF_t\right)^\thicksim
$$
for all $i \ge 0$.
Our main object of study is the following functor.
For a given coherent sheaf $\FF$ on $X$ that is $S$-flat and a tuple of functions $\mathbf{h} = (h_0,\ldots,h_r) : \ZZ^{r+1} \rightarrow \NN^{r+1}$, we consider the following functor for any scheme $T$, 
$$
\fFDir_\FF^{\mathbf{h}}(T) \;:=\;  \left\lbrace \text{morphism } T \rightarrow S \;
\begin{array}{|l}
	R^i{{f_{(T)}}_*}\left(\FF_T(\nu)\right) \text{ is locally free over $T$ and} \\ 
	\text{$\dim_{\kappa(t)}\left(\HH^i\left(X_t,\FF_t(\nu)\right)\right) = h_i(\nu)$}\\
	\text{for all $0 \le i \le r, \nu \in \ZZ, t \in T$}
\end{array}
\right\rbrace.
$$
This is a functor from (locally Noetherian) $S$-schemes to sets.
As a consequence of \autoref{prop_base_change_sheaf}, a morphism $T \rightarrow S$  belongs to the set $\fFDir_\FF^{\mathbf{h}}(T)$ if and only if $R^i{{f_{(T)}}_*}\left(\FF_T(\nu)\right)$ is a locally free $\OO_T$-module of rank $h_i(\nu)$ for all $0 \le i \le r, \nu \in \ZZ$.
The following theorem yields the representability of the functor $\fFDir_\FF^{\mathbf{h}}$.
This result will be our main tool.

\begin{theorem}
	\label{thm_flat_Sheaf}
	Let $\FF$ be a coherent sheaf on $X$ that is flat over $S$, and  $\mathbf{h} = (h_0,\ldots,h_r) : \ZZ^{r+1} \rightarrow \NN^{r+1}$ be a tuple of functions.
	Then the functor $\fFDir_\FF^{\mathbf{h}}$ is  represented by a locally closed subscheme $\FDir_\FF^{\mathbf{h}} \subset S$.
		In other words, for any morphism $g: T \rightarrow S$ of schemes, each $R^i{{f_{(T)}}_*}\left(\FF_T(\nu)\right)$ is a locally free $\OO_T$-module of rank $h_i(\nu)$  if and only if $g$ can be factored as 
		$$
		T  \; \rightarrow \; \FDir_\FF^{\mathbf{h}} \; \rightarrow \; S.
		$$
\end{theorem}
\begin{proof}
	Let $S = \bigcup_{j \in J} S_j$ be an open covering of $S$ where each $S_j$ is a Noetherian affine scheme.
	The functor $\fFDir_\FF^{\mathbf{h}}$ is a Zariski sheaf and it has a Zariski covering by the open subfunctors $\{ \mathcal{G}_j\}_{j \in J}$ where 
	$$
	\GG_j(T) \;:=\;  \left\lbrace \text{morphism } T \rightarrow S_j \;
	\begin{array}{|l}
		R^i{{f_{(T)}}_*}\left(\FF_T(\nu)\right) \text{ is locally free over $T$ and} \\ 
		\text{$\dim_{\kappa(t)}\left(\HH^i\left(X_t,\FF_t(\nu)\right)\right) = h_i(\nu)$}\\
		\text{for all $0 \le i \le r, \nu \in \ZZ, t \in T$}
	\end{array}
	\right\rbrace
	$$
	(see \cite[\S 8.3]{GORTZ_WEDHORN}).
	Therefore, due to \cite[Theorem 8.9]{GORTZ_WEDHORN}, in order to show that $\fFDir_\FF^\mathbf{h}$ is representable by a locally closed subscheme of $S$, it suffices to show that each $\GG_j$ is representable by a locally closed subscheme of $S_j$.
	
	As a consequence of the above reductions, we assume that $A$ is a Noetherian ring and $S=\Spec(A)$.
	Since all the conditions that we consider on $R^i{{f_{(T)}}_*}\left(\FF_T(\nu)\right)$ are local on $T$, we may restrict to an affine morphism $T = \Spec(B) \rightarrow \Spec(A)$, and we do so.
	
	Let $R := A[x_0,\ldots,x_r]$ with $\PP_A^r=\Proj(R)$ and $\mm=(x_0,\ldots,x_r) \subset R$.
	By known arguments, we can choose an integer $m \in \ZZ$ such that the following conditions are satisfied:
	\begin{enumerate}[(i)]
		\item $M := \bigoplus_{\nu \ge m} \HH^0(\PP_A^r, \FF(\nu))$ is a finitely generated graded $R$-module that is flat over $A$,
		\item $M^\thicksim \cong \FF$ and $\left(M \otimes_{A} B\right)^\thicksim \cong \FF_T$,
		\item $M \otimes_{A} B \cong \bigoplus_{\nu \ge m} \HH^0(\PP_B^r, \FF_T(\nu))$, and
		\item $\HH^i(\PP_A^r, \FF(\nu)) = 0$ for all $1 \le i \le r, \nu \ge m$
	\end{enumerate}
	(see, e.g., \cite[\S III.9]{HARTSHORNE}).
	Therefore, we obtain a short exact sequence 
	$$
	0 \rightarrow M \otimes_A B \rightarrow \bigoplus_{\nu \in \ZZ} \HH^0(\PP_B^r, \FF_T(\nu)) \rightarrow \HL^1(M \otimes_{A} B) \rightarrow 0
	$$ 
	that splits into the isomorphisms 
	$$
	M \otimes_A B \cong \bigoplus_{\nu \ge m} \HH^0(\PP_B^r, \FF_T(\nu)) \quad \text{ and } \quad \bigoplus_{\nu < m} \HH^0(\PP_B^r, \FF_T(\nu)) \cong \HL^1(M \otimes_{A} B),
	$$
	and we get the isomorphism $\HL^{i+1}(M \otimes_{A} B) \cong \bigoplus_{\nu \in \ZZ} \HH^i(\PP_B^r, \FF_T(\nu))$ for all $i \ge 1$. 
	
	We have obtained that $\HH^i(\PP_B^r, \FF_T(\nu))$ is a locally free $B$-module of rank $h_i(\nu)$ for all $i \ge 0, \nu \in \ZZ$ if and only if the following three conditions hold:
	\begin{itemize}[--]
		\item $[M \otimes_A B]_\nu$ is a locally free $B$-module of rank $h_0(\nu)$ for all $\nu \ge m$,
		\item $[\HL^1(M \otimes_A B)]_\nu$ is a locally free $B$-module of rank $h_0(\nu)$ for all $\nu < m$, and
		\item $[\HL^i(M \otimes_A B)]_\nu$ is a locally free $B$-module of rank $h_{i-1}(\nu)$ for all $i \ge 2, \nu \in \ZZ$.
	\end{itemize}  
	Let $h_0', h_0'' : \ZZ \rightarrow \NN$ be the functions 
	$$
	h_0'(\nu) := \begin{cases}
			h_0(\nu) \quad \text{ if } \nu < m\\
			0 \qquad\quad \text{ otherwise}
	\end{cases}
	\quad \text{ and } \qquad 
	h_0''(\nu) := \begin{cases}
		h_0(\nu) \quad \text{ if } \nu \ge m\\
		0 \qquad\quad \text{ otherwise,}
	\end{cases}
	$$
	and $\mathbf{h}' : \ZZ^{r+2} \rightarrow \NN^{r+2}$ be the tuple of functions defined by  $\mathbf{h}' := (0, h_0',h_1,\ldots,h_r)$, where $0 : \ZZ \rightarrow \NN$ denotes the zero function.
	
	Finally, by \autoref{thm_stratify_mod} and \autoref{thm_flat_Loc}, we obtain that each $\HH^i(\PP_B^r, \FF_T(\nu))$ is a locally free $B$-module of rank $h_i(\nu)$ if and only if the morphism $g : T = \Spec(B) \rightarrow S = \Spec(A)$ factors through the locally closed subscheme $F_M^{h_0''} \cap \FLoc_{M}^{\mathbf{h}'} \subset S = \Spec(A)$. 
	This concludes the proof of the theorem.
\end{proof}

\section{Fiber-full sheaves} \label{section_fiber_full_sheaf}

In this short section, we introduce the notion of fiber-full sheaf that extends the concept of fiber-full modules from \cite{FIBER_FULL}.
Let $S$ be a locally Noetherian scheme, $f : X \subset \PP_S^r \rightarrow S$ be a projective morphism, and $\FF$ be a coherent sheaf on $X$. 

\begin{definition}
	We say that $\FF$ is a \emph{fiber-full sheaf over $S$} if $R^i{{f}_*}\left(\FF(\nu)\right)$ is locally free over $S$ for all $0 \le i \le r$ and $\nu \in \ZZ$.
\end{definition}

For every $s \in S$ and $q \ge 1$, let $g_{s,q}$ be the natural map $g_{s,q} : \Spec(\OO_{S,s}/ \mm_s^q) \rightarrow S$ where $\mm_s$ denotes the maximal ideal of the local ring $\OO_{S,s}$, $X_{s,q}$ be the scheme $X_{s,q} := X \times_S \Spec(\OO_{S,s}/ \mm_s^q)$, and $\FF_{s,q}:= (1 \times_S g_{s,q})^*\FF$ be the sheaf on $X_{s,q}$ obtained by the pull-back induced by $g_{s,q}$.
For the case $q=1$ (i.e., when we take the fiber at a point $s \in S$), we simply write $g_s = g_{s,1}$, $X_s = X_{s,1}$ and $\FF_s = F_{s,1}$.
The following theorem gives two further equivalent definitions for the notion of a fiber-full sheaf. 
The name ``fiber-full'' is inspired by condition (3) below.

\begin{theorem}
	\label{thm_fib_full_sheaf}
	The following three conditions are equivalent:
	\begin{enumerate}[\rm (1)]
		\item $\FF$ is a fiber-full sheaf over $S$.
		\item $\FF$ is a locally free $\OO_S$-module and $\HH^i(X_{s,q}, \FF_{s,q}(\nu))$ is a free $\OO_{S,s}/ \mm_s^q$-module for all $s \in S$, $0 \le i \le r$, $\nu \in \ZZ$ and $q \ge 1$.
		\item $\FF$ is a locally free $\OO_S$-module and the natural map $\HH^i(X_{s,q}, \FF_{s,q}(\nu)) \rightarrow \HH^i(X_{s}, \FF_s(\nu))$ is surjective for all $s \in S$, $0 \le i \le r$, $\nu \in \ZZ$ and $q \ge 1$.
	\end{enumerate}
\end{theorem}
\begin{proof}
	Since the three conditions are local on $S$, we can choose a point $s \in S$ and assume  that $(B,\bb) = (\OO_{S,s}, \mm_s)$ is a Noetherian local ring and $S = \Spec(B)$.
	Moreover, in each of the three above conditions one is assuming that $\FF$ is flat over $S$.
	Let $R := B[x_0,\ldots,x_r]$ with $\PP_B^r=\Proj(R)$ and $\mm=(x_0,\ldots,x_r) \subset R$.
	Then we can choose an integer $m \in \ZZ$ such that the following conditions are satisfied:
	\begin{enumerate}[(i)]
		\item $M := \bigoplus_{\nu \ge m} \HH^0(\PP_B^r, \FF(\nu))$ is a finitely generated graded $R$-module that is flat over $B$,
		\item $M^\thicksim \cong \FF$ and $\left(M \otimes_{B} B/\bb^q\right)^\thicksim \cong \FF_{s,q}$, and
		\item $M \otimes_{B} B/\bb^q \cong \bigoplus_{\nu \ge m} \HH^0(\PP_{B/\bb^q}^r, \FF_{s,q}(\nu))$.
	\end{enumerate}
	Similar to the proof of \autoref{thm_flat_Sheaf}, by using the relations between local and sheaf cohomologies,  the equivalence of the three conditions follows directly from \cite[Theorem A]{FIBER_FULL}.
\end{proof}

\section{Construction of the fiber-full scheme} \label{section_fiber_full_scheme}

In this section, we construct the \emph{fiber-full scheme} which can be seen as a parameter space that generalizes the Hilbert and Quot schemes and that controls all the cohomological data instead of just the corresponding Hilbert polynomial. 
We also construct open subschemes of the fiber-full scheme that parametrize arithmetically Cohen-Macaulay and arithmetically Gorenstein schemes.
We use the following setup throughout this section. 

\begin{setup}
	\label{setup_fib_full}
	Let $S$ be a locally Noetherian scheme, $f : X \subset \PP_S^r \rightarrow S$ be a projective morphism, and $\FF$ be a coherent sheaf on $X$. 
\end{setup}

We define the \emph{fiber-full functor} which for an $S$-scheme $T$ parametrizes all coherent quotients $\FF_T \twoheadrightarrow \GG$ such that all higher direct images of $\GG$ and its twistings are locally over $T$.
That is, we define the following map for any (locally Noetherian)  $S$-scheme $T$:
$$
\fFib_{\FF/X/S}(T) \;:=\; \left\lbrace \text{coherent quotient $\FF_T \twoheadrightarrow \GG$} \;
\begin{array}{|l}
	 R^i{{f_{(T)}}_*}\left(\GG(\nu)\right) \text{ is locally free over $T$}\\
	 \text{for all $0 \le i \le r, \nu \in \ZZ$}
\end{array}
 \right\rbrace.
$$
One important basic thing about this map is the next lemma, which tells us that 
$$
\fFib_{\FF/X/S} \;: \; (\text{Sch}/S)^\text{opp} \rightarrow (\text{Sets})
$$ 
is a contravariant functor from the category of (locally Noetherian)  $S$-schemes to the category of sets.

\begin{lemma} 
	\label{lem_base_change_Fib}
	Let $g : T' \rightarrow T$ be morphism of (locally Noetherian)  $S$-schemes.
	Then we have a natural map
	$$
	 \fFib_{\FF/X/S}(g)  \;:\; \fFib_{\FF/X/S}(T) \rightarrow \fFib_{\FF/X/S}(T'), \qquad \GG \mapsto (1 \times_T g)^*\GG,
	$$
	where $(1 \times_T g)^* \GG$ is the sheaf on $X_{T'}$ obtained by the pull-back induced by $g$.
\end{lemma}
\begin{proof}
	This is a direct consequence of \autoref{prop_base_change_sheaf}.
\end{proof}

We now stratify this functor in terms of ``Hilbert functions'' for all the cohomologies.
Let $\mathbf{h} = (h_0,\ldots,h_r) : \ZZ^{r+1} \rightarrow \NN^{r+1}$ be a tuple of functions.
Then we define the following functor depending on $\mathbf{h}$: 
$$
\fFib^{\mathbf{h}}_{\FF/X/S}(T) \;:=\; 
\left\lbrace \GG \in \fFib_{\FF/X/S}(T) \;\;
\begin{array}{|l}
	\text{$\dim_{\kappa(t)}\left(\HH^i\left(X_t,\GG_t(\nu)\right)\right) = h_i(\nu)$}\\
	\text{for all $0 \le i \le r, \nu \in \ZZ, t \in T$}
\end{array}
\right\rbrace.
$$
The idea of this functor is to measure the dimension of \emph{all cohomologies of all possible twistings}.
Of course, we obtain the  following stratification 
$$
\fFib_{\FF/X/S}(T) \;=\; \bigsqcup_{\mathbf{h} : \ZZ^{r+1} \rightarrow \NN^{r+1}} \fFib^{\mathbf{h}}_{\FF/X/S}(T)
$$
when $T$ is connected.
Therefore, $\fFib_{\FF/X/S}(T)$ is a representable functor if all the functors $\fFib^{\mathbf{h}}_{\FF/X/S}(T)$ are representable.
When $\FF = \OO_X$, we simplify the notation by writing $\fFib^{\mathbf{h}}_{X/S}$, and we obtain the following alternative description of significant interest
$$
\fFib_{X/S}^\mathbf{h}(T) \;:= \;  \left\lbrace \text{closed subscheme } Z \subset X_T \;
\begin{array}{|l}
	R^i{{f_{(T)}}_*}\left(\OO_Z(\nu)\right) \text{ is locally free over $T$ and} \\ 
	\text{$\dim_{\kappa(t)}\left(\HH^i\left(Z_t,\OO_{Z_t}(\nu)\right)\right) = h_i(\nu)$}\\
	\text{for all $0 \le i \le r, \nu \in \ZZ, t\in T$}
\end{array}
\right\rbrace.
$$
These functors should be thought of as a refinement of the Hilbert and Quot functors in the following sense.

\begin{remark}
	\label{rem_relation_Fib_Hilb}
	Let $\mathbf{h} = (h_0,\ldots,h_r) : \ZZ^{r+1} \rightarrow \NN^{r+1}$ be a tuple of functions and suppose that $P_\mathbf{h} := \sum_{i=0}^{r} (-1)^i h_i \in \QQ[m]$ is a numerical polynomial.
	Then we automatically obtain the following inclusions 
	$$
	\fFib^{\mathbf{h}}_{X/S}(T) \; \subset \; \fHilb^{P_\mathbf{h}}_{X/S}(T) \quad \text{ and } \quad \fFib^{\mathbf{h}}_{\FF/X/S}(T) \; \subset \; \fQuot^{P_\mathbf{h}}_{\FF/X/S}(T).
	$$
	We say that $P_\mathbf{h}$ is the Hilbert polynomial corresponding with the prescribed ``Hilbert functions'' $\mathbf{h} : \ZZ^{r+1} \rightarrow \NN^{r+1}$ of cohomologies.
	If the function $P_\mathbf{h} = \sum_{i=0}^{r} (-1)^i h_i$ does not coincide with a numerical polynomial then $\fFib^{\mathbf{h}}_{X/S}(T) = \emptyset$ for all $S$-schemes $T$.
\end{remark}

Our main result is the following theorem which says that the functor $\fFib^{\mathbf{h}}_{\FF/X/S}$ is represented by a quasi-projective $S$-scheme.

\begin{theorem}
	\label{thm_main}
	Let $\mathbf{h} = (h_0,\ldots,h_r) : \ZZ^{r+1} \rightarrow \NN^{r+1}$ be a tuple of functions and suppose that $P_\mathbf{h}(m) \in \QQ[m]$ is a numerical polynomial.
	Then there is a quasi-projective $S$-scheme $\Fib^{\mathbf{h}}_{\FF/X/S}$ that represents the functor $\fFib^{\mathbf{h}}_{\FF/X/S}$ and that is a locally closed subscheme of the Quot scheme $\Quot^{P_\mathbf{h}}_{\FF/X/S}$.
\end{theorem}
\begin{proof}
	By \autoref{rem_relation_Fib_Hilb}, there is an injective morphism of functors 
	$$
	\Phi \,:\,  \fFib^{\mathbf{h}}_{\FF/X/S} \rightarrow  \fQuot^{P_\mathbf{h}}_{\FF/X/S}.
	$$
	We shall show that $\fFib^{\mathbf{h}}_{\FF/X/S}$ is a locally closed subfunctor of $\fQuot^{P_\mathbf{h}}_{\FF/X/S}$.
	By the existence of the Quot scheme \cite{GROTHENDIECK_HILB, ALTMAN_KLEIMAN_COMPACT_PICARD}, the functor $\fQuot^{P_\mathbf{h}}_{\FF/X/S}$ is represented by a projective $S$-scheme $\Quot^{P_\mathbf{h}}_{\FF/X/S}$  and a universal quotient  $\FF_{\Quot^{P_\mathbf{h}}_{\FF/X/S}} \twoheadrightarrow \mathcal{W}^{P_\mathbf{h}}_{\FF/X/S}$ in $\fQuot^{P_\mathbf{h}}_{\FF/X/S}\big(\Quot^{P_\mathbf{h}}_{\FF/X/S}\big)$.
	Let $Q := \Quot^{P_\mathbf{h}}_{\FF/X/S}$ and $\mathcal{W} := \mathcal{W}^{P_\mathbf{h}}_{\FF/X/S}$.
	Thus, for each $S$-scheme $T$ and for each quotient $\FF_T \twoheadrightarrow \GG$ in  $\fQuot^{P_\mathbf{h}}_{\FF/X/S}(T)$, there is a unique classifying $S$-morphism $g_{T,\GG} : T \rightarrow Q$ such that 
	$
	\GG  =  (1 \times_S g_{T,G})^*\mathcal{W}.
	$
	
	By using \autoref{thm_flat_Sheaf}, let $\Fib^{\mathbf{h}}_{\FF/X/S} := \FDir_\mathcal{W}^{\mathbf{h}} \subset Q$ be the locally closed subscheme of $Q$ that represents the functor $\fFDir_\mathcal{W}^{\mathbf{h}}$.
	So, it follows that a quotient in $\FF_T \twoheadrightarrow \GG$ in  $\fQuot^{P_\mathbf{h}}_{\FF/X/S}(T)$ belongs to $\fFib^{\mathbf{h}}_{\FF/X/S}(T)$ if and only if $g_{T,\GG}$ factors through $\Fib^{\mathbf{h}}_{\FF/X/S}$.
	Finally, this shows that the functor $\fFib^{\mathbf{h}}_{\FF/X/S}$ is represented by the $S$-scheme $\Fib^{\mathbf{h}}_{\FF/X/S}$ and by the universal quotient $\FF_{\Fib^{\mathbf{h}}_{\FF/X/S}} \twoheadrightarrow (1 \times_S \iota)^*\mathcal{W}$ in $\fFib^{\mathbf{h}}_{\FF/X/S}\big(\Fib^{\mathbf{h}}_{\FF/X/S}\big)$, where $\iota : \Fib^{\mathbf{h}}_{\FF/X/S} \hookrightarrow Q$ denotes the natural locally closed immersion.
	Since $Q$ is a projective $S$-scheme, we obtain that $\Fib^{\mathbf{h}}_{\FF/X/S}$ is a quasi-projective $S$-scheme.
\end{proof}

\begin{remark}
	\label{rem_subfunctor}
	As pointed out in the proof \autoref{thm_main}, $\fFib^{\mathbf{h}}_{\FF/X/S}$ is a locally closed subfunctor of $\fQuot^{P_\mathbf{h}}_{\FF/X/S}$.
\end{remark}

\begin{remark}
	When the base scheme $S$ is obvious from
	context, we may simply write the fiber-full schemes as $\Fib^{\mathbf{h}}_{\FF/X}$ and $\Fib^{\mathbf{h}}_{X}$ instead of $\Fib^{\mathbf{h}}_{\FF/X/S}$ and $\Fib^{\mathbf{h}}_{X/S}$, respectively.
\end{remark}

\begin{remark} Since the dimensions of the cohomology groups can jump in flat families, the fiber-full scheme is usually not projective \cite[Example III.12.9.2]{HARTSHORNE}.
\end{remark}

We now recall the following notions. 

\begin{definition}
	Let $\kk$ be a field and $Y \subset \PP_\kk^r$ be a closed subscheme.
	Let $R_Y$ be the homogeneous coordinate ring of $Y$.
	We say that $Y$ is \emph{arithmetically Cohen-Macaulay} (ACM for short) if $R_Y$ is a Cohen-Macaulay ring. 
	If $R_Y$ is a Gorenstein ring then $Y$ is said to be \emph{arithmetically Gorenstein} (AG for short).
\end{definition}

Next, we show the existence of open subschemes of the fiber-full scheme that parametrize ACM and AG schemes.
A closed subscheme $Y \subset \PP_\kk^r$ is ACM if and only if the following two conditions are satisfied:
\begin{enumerate}[(1)]
	\item $\HH^i(Y, \OO_Y(\nu)) =0$ for all $1 \le i \le \dim(Y)-1$ and $\nu \in \ZZ$, and
	\item the natural map $R_Y \rightarrow \bigoplus_{\nu \in \ZZ}\HH^0(Y, \OO_Y(\nu))$ is bijective if $\dim(Y) > 0$, or injective if $\dim(Y)=0$. 
\end{enumerate}
Let $0 \le d \le r$ and $h_0, h_d : \ZZ \rightarrow \NN$ be two functions, and consider the tuple of functions $\mathbf{h} : \ZZ^{r+1} \rightarrow \NN^{r+1}$ given by $\mathbf{h} = (h_0, 0,\ldots,0,h_d,0,\ldots,0)$ where $0 : \ZZ \rightarrow \NN$ denotes the zero function.
To study ACM and AG schemes, it then becomes natural to consider the following two functors. 
For any (locally Noetherian) $S$-scheme  $T$, we have 
 $$
 \fACM_{X/S}^{h_0,h_d}(T) \;:= \;  \left\lbrace \text{closed subscheme } Z \subset X_T \;
 \begin{array}{|l}
 	Z \in \fFib^{\mathbf{h}}_{X/S}(T) \text{ and $Z_t$ is  ACM for all $t\in T$} 
 \end{array}
 \right\rbrace
 $$
 and
 $$
 \fAG_{X/S}^{h_0,h_d}(T) \;:= \;  \left\lbrace \text{closed subscheme } Z \subset X_T \;
 \begin{array}{|l}
 	Z \in \fFib^{\mathbf{h}}_{X/S}(T) \text{ and $Z_t$ is  AG for all $t\in T$} 
 \end{array}
 \right\rbrace.
 $$
By using the base change results of \autoref{lem_Ext_base_change} and \autoref{lem_base_change_loc}, we can immediately deduce that $ \fACM_{X/S}^{h_0,h_d}$ and $ \fAG_{X/S}^{h_0,h_d}$ are indeed contravariant functors from the category of (locally Noetherian) $S$-schemes into the category of sets.
 The following theorem gives the representability of these two functors.

 \begin{theorem} \label{thm_ACM}
 	Let $0\le d \le r$ and $h_0, h_d : \ZZ \rightarrow \NN$ be two functions, and consider the tuple of functions $\mathbf{h}  = (h_0, 0,\ldots,0,h_d,0,\ldots,0): \ZZ^{r+1} \rightarrow \NN^{r+1}$.
 	Suppose that $P_\mathbf{h}(m) \in \QQ[m]$ is a numerical polynomial.
 	Then there exist open $S$-subschemes $\ACM_{X/S}^{h_0,h_d}$ and $\AG_{X/S}^{h_0,h_d}$ of $\Fib^{\mathbf{h}}_{X/S}$ that represent the functors  $\fACM_{X/S}^{h_0,h_d}$ and $\fAG_{X/S}^{h_0,h_d}$, respectively.
 \end{theorem}
\begin{proof}
	By \autoref{thm_main},  there is a pair $(\Fib^{\mathbf{h}}_{\FF/X/S}, \mathcal{I})$  representing the functor $\fFib^{\mathbf{h}}_{\FF/X/S}$, where $\Fib^{\mathbf{h}}_{\FF/X/S}$ is fiber-full scheme and $ \mathcal{I}$ is the universal ideal sheaf on $\PP_{\Fib^{\mathbf{h}}_{\FF/X/S}}^r$.
	Let $F := \Fib^{\mathbf{h}}_{\FF/X/S}$.
	This means that, for each $S$-scheme $T$ and for each $Z \in \fFib^{\mathbf{h}}_{\FF/X/S}(T)$, there is a unique classifying $S$-morphism $g_{T, Z} : T \rightarrow F$ such that  $\mathcal{I}_{Z}=(1 \times_S g_{T,Z})^*\mathcal{I}$ is the ideal sheaf on $\PP_T^r$ that corresponds with the closed subscheme $Z \subset \PP_T^r$.
	
	Fix $Z \in \fFib^{\mathbf{h}}_{\FF/X/S}(T)$, $g_{T, Z} : T \rightarrow F$ and $\mathcal{I}_{Z}=(1 \times_S g_{T,Z})^*\mathcal{I}$.
	Since the conditions defining the functors $\fACM_{X/S}^{h_0,h_d}$ and $\fAG_{X/S}^{h_0,h_d}$ are local on $T$, we can restrict the morphism $g_{T,Z}$ to affine open subschemes $\Spec(B) \subset T$ and $\Spec(A) \subset F$ with $A$ being Noetherian.
	So, we assume that $T = \Spec(B)$ and $F=\Spec(A)$.
	Let $R := A[x_0,\ldots,x_r]$ with $\PP_A^r=\Proj(R)$ and $\mm=(x_0,\ldots,x_r) \subset R$.
	Let $I \subset R$ be the saturated ideal $I:=\bigoplus_{\nu \in \ZZ} \HH^0(\PP_A^r, \mathcal{I}(\nu))$.
	The saturated ideal and homogeneous coordinate ring of $Z$ are given by $I_Z := \bigoplus_{\nu \in \ZZ} \HH^0(\PP_B^r, \mathcal{I}_Z(\nu)) \cong I \otimes_{A} B$ and $R_Z := B[x_0,\ldots,x_r]/I_Z \cong  R/I \otimes_{A} B$, respectively. 
	For all $t \in T$, let $R_t := B[x_0,\ldots,x_r] \otimes_{B} \kappa(t) \cong R \otimes_{A} \kappa(t)$ and $R_{Z,t}  := R_Z \otimes_{B} \kappa(t) \cong R/I \otimes_{A} \kappa(t)$.
	
	First, we show that $\fACM_{X/S}^{h_0,h_d}$ is representable.
	By construction, $\HL^0(R_{Z,t})=0$ for all $t \in T$, and so $Z_t$ is ACM  for all $t \in T$ when $d =0$. 
	If $d>0$, we have that $Z_t$ is ACM for all $t \in T$  if and only if $\HL^1(R_{Z,t}) =0$ for all $t \in T$.
	We have that the locus $V:=\{ f \in F \mid \HL^1(R/I \otimes_{A} \kappa(f)) = 0\}$ is an open subscheme of $F$.
	When $d > 0$,  $g_{T,Z} : T = \Spec(B) \rightarrow F = \Spec(A)$ factors through $V$ if and only if $Z \in \fACM_{X/S}^{h_0,h_d}(T)$.
	Therefore, it follows that, in both cases $d =0$ or $d>0$, $\fACM_{X/S}^{h_0,h_d}$	is represented by an open subscheme of $\ACM_{X/S}^{h_0,h_d} \subset F$.
	
	We now concentrate on the representability of 	 $\fAG_{X/S}^{h_0,h_d}$.
	Since a Gorenstein ring is Cohen-Macaulay, we assume that $Z \in \fACM_{X/S}^{h_0,h_d}(T)$ and so $g_{T,Z}$ factors through $\ACM_{X/S}^{h_0,h_d} \subset F$.
	Therefore, as $R_{Z,t}$ is a Cohen-Macaulay ring of dimension $d+1$, it is Gorenstein if and only if its $(d+1)$-th Bass number
	$$
	 \mu_{d+1}(R_{Z,t}) \;:=\; \dim_{\kappa(t)} \big(\Ext_{R_{Z,t}}^{d+1}(R_{Z,t}/\mm R_{Z,t}, R_{Z,t})\big) 
	$$
	is equal to one (see \cite[Theorem 3.2.10]{BRUNS_HERZOG}).
	By upper semicontinuity, the locus
	$$
	W:=\big\{ f \in F \mid  \mu_{d+1}(R/I \otimes_{A} \kappa(f)) \le 1\big\}
	$$
	is an open subscheme of $F$.
	On the other hand, if $f \in \ACM_{X/S}^{h_0,h_d}$, then $ \mu_{d+1}(R/I \otimes_{A} \kappa(f)) \ge 1$.
	Finally, it follows that $g_{T,Z} : T = \Spec(B) \rightarrow F = \Spec(A)$ factors through $\AG_{X/S}^{h_0,h_d}:= \ACM_{X/S}^{h_0,h_d} \cap W$ if and only if $Z \in \fAG_{X/S}^{h_0,h_d}(T)$.
	So, the proof of the theorem is complete.
\end{proof}

\section{Examples} \label{section_examples_1}
In this section, we study some examples of fiber-full schemes as subschemes of the Hilbert scheme. These include the Hilbert scheme of points and classical examples such as the Hilbert scheme of twisted cubics and the Hilbert scheme of skew lines.

\begin{example}[Points] \label{0_DIM} 
Let $S$ be a locally Noetherian scheme and $f:X \subseteq \PP^r_S \to S$ be a projective morphism. 
Let $\mathbf{h}: \Z^{r+1} \to \NN^{r+1}$ be the tuple of constant functions defined by $\mathbf{h} := (c,0,\dots,0)$  and let $P_{\mathbf{h}} = c$ be the associated Hilbert polynomial. 
For any $S$-scheme $T$ and $Z \in \fHilb_{X/S}^{P_\mathbf{h}}(T)$, we have 
$$
\dim_{\kappa(t)} \left( \HH^i(Z_t,\OO_{Z_t}(\nu)) \right) = 
	\begin{cases}
	 c & \text{ if } i = 0 \\
	0 & \text{ if } i > 0
	\end{cases}
$$
for all $t \in T$ and  $\nu \in \ZZ$. 
It follows that 
$
\fFib_{X/S}^\mathbf{h}(T) = \fHilb_{X/S}^{P_{\mathbf{h}}}(T)
$
for all $T$ and thus 
$$
\Fib_{X/S}^\mathbf{h} = \Hilb_{X/S}^{P_{\mathbf{h}}}.
$$
In particular, $\Fib^{\mathbf{h}}_{\PP^r}$ satisfies Murphy's law up to retraction for $r \geq 16$ \cite[Theorem 1.3]{JJ}. More generally, for any coherent sheaf $\FF$ on $X$, we have $\Fib_{\FF/X/S}^\mathbf{h} = \Quot_{\FF/X/S}^{P_{\mathbf{h}}}$.
\end{example}

For the next two examples, let $\kk$ be an algebraically closed field of characteristic zero.

\begin{example}[Twisted cubics] \label{TWISTED_CUBIC} By the work of \cite{TWISTED_CUBIC} it is known that $\Hilb^{3m+1}_{\PP^3_{\kk}} = H \cup H'$ is a union of two smooth irreducible components such that the general member of $H$ parametrizes a twisted cubic, and the general member of $H'$ parametrizes a plane cubic union an isolated point. It is also known that $H - H \cap H'$ is the locus of arithmetically Cohen-Macaulay curves of degree $3$ and genus $0$.
Then we have a decomposition
$$
\Hilb^{3m+1}_{\PP^3_{\kk}} = \text{Fib}^{(\mathbf{h},0,0)}_{\PP^3_{\kk}} \sqcup \text{Fib}^{(\mathbf{h}',0,0)}_{\PP^3_{\kk}} =  (H - H \cap H') \sqcup H'
$$
where $\mathbf{h} =(h_0,h_1),\mathbf{h}'=(h_0',h_1'): \ZZ^2 \rightarrow \NN^2$ are the tuples of functions given by 
$$
h_0(\nu) = \begin{cases} 
	0 & \text{ if } \nu \leq -1 \\
	3v+1 & \text{ if } \nu \geq 0,
	\end{cases}
 \quad
h_0'(\nu) = \begin{cases} 
	1 & \text{ if } \nu \leq -1 \\
	2 & \text{ if } \nu = 0 \\
	3v+1& \text{ if } \nu \geq 1
	\end{cases}
	\quad
	\text{and}
	\quad
	\begin{array}{l}
		h_1(\nu) = h_0(\nu)-(3\nu+1)\\
		h_1'(\nu) = h_0'(\nu)-(3\nu+1).
	\end{array}
$$
To verify this decomposition we appeal to the classification of ideals in \cite[\S 4]{TWISTED_CUBIC}. Since 
$$
h^0(X,\OO_X(\nu)) = \chi(\OO_X(\nu)) + h^1(X,\OO_X(\nu)) = 3\nu+1 + h^1(X,\OO_X(\nu))
$$ 
for any $[X] \in \Hilb^{3m+1}_{\PP^3_{\kk}}$, it suffices to compute $h^0(X,\OO_X(\nu))$. 
Any subscheme $[X] \in H-H \cap H'$ is arithmetically Cohen-Macaulay with the ideal sheaf having a resolution
$$
0 \to \OO_{\PP_\kk^3}(-3)^2 \to \OO_{\PP_\kk^3}(-2)^3 \to \mathcal{I}_{X} \to 0.
$$
It follows that  $h^0(\mathcal{I}_X(\nu)) = 3\binom{\nu+1}{3} - 2\binom{\nu}{3}$. Using the ideal sheaf exact sequence and the fact that $h^1(\mathcal{I}_X(\nu)) = 0$ we deduce that $h^0(X,\OO_X(\nu)) = \binom{\nu+3}{3} -  3\binom{\nu+1}{3} + 2\binom{\nu}{3} = 3\nu+1$ for $\nu \geq 0$ and $0$ otherwise, as required.

If $[X] \in H'$ then $\mathcal{I}_X = \mathcal{I}_{X'} \cap \mathcal{J}$ where $X'$ is a plane cubic and $\mathcal{J}$ defines a, possibly embedded, $0$-dimensional subscheme. Consider the exact sequence
$$
0 \to \mathcal{I}_{X'}/\mathcal{I}_{X} \to \OO_X \to \OO_{X'} \to 0
$$
of sheaves on $X$. Since $\mathcal{I}_{X'}/\mathcal{I}_{X}$ is $0$-dimensional, we have 
$
h^0(X,\mathcal{I}_{X'}/\mathcal{I}_{X}) =  \text{length}(\mathcal{I}_{X'}/\mathcal{I}_{X} ) = (3m+1) - 3m = 1.
$ 
It is straightforward to show that the cohomology of a plane curve $Y$ of degree $d$ is given by $h^0(Y,\OO_Y(\nu)) = \binom{\nu+2}{2} - \binom{\nu+2-d}{2}$. 
Thus, we deduce that $h^0(X,\OO_X(\nu)) = h^0(X',\OO_{X'}(\nu)) + 1 = \binom{\nu+2}{2}-\binom{\nu-1}{2} + 1$, as required. 
\end{example}

\begin{example}[Skew Lines] The Hilbert scheme $\Hilb^{2m+2}_{\PP^3_{\kk}}$ has been studied in \cite{CODIM_2} and \cite{LINEAR_SPACES}. In parallel with the Hilbert scheme of the twisted cubic, it is known that $\Hilb^{2m+2}_{\PP^3_{\kk}} = H \cup H'$ is a union of smooth components such that the general member of $H$ parametrizes two skew lines, and the general member of $H'$ parametrizes a plane conic union an isolated point. Furthermore $H$ is stratified by $\{H_1,\dots,H_4\}$ where 
	\begin{enumerate}
		\item $H_1 = \GL(4) \cdot V((x_0,x_1) \cap (x_2,x_3))$ is the locus of skew lines, 
		\item $H_2 = \GL(4) \cdot V((x_0,x_1) \cap (x_0,x_2) \cap (x_0^2,x_1,x_2))$ is the locus of lines meeting along an embedded point,
		\item $H_3 = \GL(4) \cdot V(x_0^2,x_0x_1,x_1^2,x_0x_2-x_1x_3)$ is the locus of pure double structures on a line,
		\item $H_4 = \GL(4) \cdot V((x_0,x_1^2) \cap (x_0^2,x_1,x_2))$ is the locus of double structures with an embedded point.
	\end{enumerate}
We also have $H \cap H' = H_2 \cup H_4 = \overline{H_2}$ (see \cite[Theorem 1.1(2)]{CODIM_2}, \cite[Example 2.25]{LINEAR_SPACES}).		
Then we have a decomposition
$$
\Hilb^{2m+2}_{\PP^3_{\kk}} = \text{Fib}^{(\mathbf{h},0,0)}_{\PP^3_{\kk}} \sqcup \text{Fib}^{(\mathbf{h}',0,0)}_{\PP^3_{\kk}} =  (H - H \cap H') \sqcup  H'  
$$
where $\mathbf{h}= (h_0,h_1), \mathbf{h'} = (h'_0,h'_1) : \ZZ^2 \rightarrow \NN^2$ are the tuples of functions 
$$
	h_0(\nu) = \begin{cases} 
	0 & \text{ if } \nu \leq -1 \\
	2\nu+2 & \text{ if } \nu \geq 0,
	\end{cases}
\quad 
	h_0'(\nu) = \begin{cases} 
	1 & \text{ if } \nu \leq -1 \\
	2v+2& \text{ if } \nu \geq 0
	\end{cases}
\quad
\text{and}
\quad
	\begin{array}{l}
	h_1(\nu) = h_0(\nu)-(2\nu+2)\\
	h_1'(\nu) = h_0'(\nu)-(2\nu+2).
\end{array}
$$
To prove this we proceed as in \autoref{TWISTED_CUBIC}. 
If $[X] \in H'$, we have $\mathcal{I}_X = \mathcal{I}_{X'} \cap \mathcal{J}$ where $X'$ is a plane conic  and $\mathcal{J}$ defines a $0$-dimensional subscheme. Arguing as in  \autoref{TWISTED_CUBIC}, we deduce that $h^i(X,\OO_X(\nu)) = h'_i(\nu)$ for all $\nu,i$. 
Note that $H- \overline{H_2} = H_1 \cup H_3$. 
If $[X] \in H_1$ is a pair of skew lines, we have $h^i(X,\OO_X(\nu)) = 2h^i(\PP_\kk^1,\OO_{\PP_\kk^1}(\nu)) = h_i(\nu)$ for all $\nu,i$. 
If $[X] \in H_3$, then $X$ defines a ribbon on $\PP_\kk^1$ of genus $-1$ \cite[\S 2]{RIBBONS} and there is an exact sequence
$$
0 \to \OO_{\PP_\kk^1} \to \OO_X \to \OO_{\PP_\kk^1} \to 0
$$
of sheaves on $X$. It follows that the $h^i(X,\OO_X(\nu)) = h^i(\PP_\kk^1,\OO_{\PP_\kk^1}(\nu)^{\oplus 2}) = h_i(\nu)$ for all $\nu,i$, as required.
\end{example}

\section{Smooth Hilbert schemes} \label{section_examples_2}

In this section, we study the fiber-full scheme as a subscheme of smooth Hilbert schemes, the latter were recently classified in \cite{SMOOTH_HILB}. 
Our main result states that if the Hilbert scheme is smooth, then it is equal to a fiber-full scheme.
For the convenience of the reader, we review some of the discussions from \cite[\S 3]{SMOOTH_HILB}. 
For the rest of this section, we use the following setup.

\begin{setup}
	\label{setup_smooth_Hilb}
	Let $\kk$ be a field, $R = \kk[x_0,\dots,x_r]$ be a polynomial ring and $\PP_\kk^r = \Proj(R)$. 
\end{setup} 

A numerical polynomial $P\in \QQ[m]$ of degree at most $r$ is the Hilbert polynomial of some subscheme in $\PP_\kk^r$ if and only if there exists an \textit{integer partition} $\lambda = (\lambda_1,\dots,\lambda_n)$ of positive integers such that $\lambda_1 \geq \cdots 
\geq \lambda_n$ and 
$$
 P = P_{\lambda} := \sum_{i=1}^n \binom{m+\lambda_i-i}{\lambda_i-1} \,\in\, \QQ[m].
$$
The dimension of any subscheme with Hilbert polynomial $P_\lambda$ is $\lambda_1-1$. 

For an integer partition $\lambda$, there is a unique saturated lexicographic ideal, denoted by $L(\lambda)$, with Hilbert polynomial $P_{\lambda}$. To describe this, let $a_j$ be the number of parts in $\lambda$ equal to $j$ for all $j \in \NN$. If $r \geq \lambda_1$ we have
$$
L(\lambda) = L(a_1,\dots,a_r) := (x_0^{a_r+1},x_0^{a_r}x_1^{a_{r-1}+1}, \dots,
	x_0^{a_r}x_1^{a_{r-1}}\cdots x_{r-3}^{a_3}x_{r-2}^{a_2+1},
	x_0^{a_r}x_1^{a_{r-1}} \cdots a_{r-2}^{a_2}x_{r-1}^{a_1}).
$$
The only case not accounted by the condition $r \geq \lambda_1$ is $\lambda = (r+1)$; in this case the corresponding lexicographic ideals is $(0)$.
For proofs of these facts see \cite[Lemma 3.3, Proposition 3.5]{SMOOTH_HILB}.

\begin{definition} For an integer partition $\lambda$, define the tuple of functions $\mathbf{h}_{\lambda} =(h_0,\dots,h_{r}):\ZZ^{r+1} \to \NN^{r+1}$ given by $h_i(\nu) := \dim_{\kk} \left(\HH^i(\PP_\kk^r, \OO_{V(L(\lambda))}(\nu))\right)$ for all $\nu \in \ZZ$.
\end{definition}

We begin by describing $\mathbf{h}_\lambda$ explicitly.

\begin{lemma} \label{lemma_lex_ext} Let $\lambda = (\lambda_1,\dots,\lambda_n) \ne (r+1)$ be an integer partition and $L(\lambda) = L(a_1\dots,a_{r})$ be the associated lexicographic ideal. 
Then for all $\nu\in \ZZ$ we have
$$
\dim_{\kk }\left(\HH^i(\PP_\kk^r, \OO_{V(L(\lambda))}(\nu))\right) = \begin{cases}
	\sum_{i=1}^n \binom{\nu+\lambda_i-i}{\nu-i+1} +
			     		\binom{a_{1} + \cdots +a_{r} - \nu-1}{1} - \binom{a_{2} + \cdots +a_{r} -\nu-1}{1} 
						& \text{ if } i = 0 \\
	\binom{a_{i+1} + \cdots +a_{r} - \nu-1}{i+1} - \binom{a_{i+2} + \cdots + a_{r} -\nu-1}{i+1} & \text{ if } i > 0.
	\end{cases}
$$
\end{lemma}
\begin{proof} 
	Fix $L = L(\lambda)$.
	By \cite[Lemma 3.2]{REEVES_STILLMAN}, we obtain
	$$
	\Ext_R^i(R/L,R) \cong \left(R/(x_0,\dots,x_{i-2},x_{i-1}^{a_{r-i+1}})\right)(a_{r-i+1} + \cdots + a_{r} + i-1), \quad 1 \leq i \leq r.
	$$
Note that $a_l$ in the notation of \cite{REEVES_STILLMAN}  corresponds to $a_{l+1}$ in our convention. Using the exact sequence 
$$0 \to (R/(x_0,\dots,x_{q-1}))(-p) \to R/(x_0,\dots,x_{q-1}) \to R/(x_0,\dots,x_{q-1},x_q^p) \to 0, 
$$ 
we deduce that
$$
\dim_\kk\left(\left[R/(x_0,\dots,x_{q-1},x_{q}^{p})\right]_{\nu}\right) = \binom{\nu + r-q}{r-q} - \binom{\nu - p+ r-q}{r-q}. 
$$
Using the above formulas and the local duality theorem (see, e.g., \cite[Theorem 3.6.19]{BRUNS_HERZOG}),  we obtain
\begin{align*}
\dim_\kk \left(\HH^{i}(\PP_\kk^r,\OO_{V(L)}(\nu))\right) & = \dim_\kk  \left([\HH_{\mm}^{i+1}(R/L)]_\nu\right) \\
			 & = \dim_\kk \left([\Ext_R^{r-i}(R/L,R)]_{-\nu-r-1}\right) \\
			 & = \dim_\kk \left(\left[R/(x_0,\dots,x_{r-i-2},x_{r-i-1}^{a_{i+1}})(a_{i+1} + \cdots + a_{r} + r-i-1)\right]_{-\nu-r-1}\right)\\
			 & =  \binom{a_{i+1} + \cdots + a_{r} - \nu-1}{i+1} - \binom{a_{i+2} + \cdots + a_{r} -\nu-1}{i+1}.
\end{align*}
for all $i>0$.
Similarly, since $L$ is saturated, we obtain
\begin{align*}
\dim_\kk\left( \HH^0(\PP_\kk^r,\OO_{V(L)}(\nu))\right) & = \dim_\kk \left( [R/L]_\nu \right) + \dim_\kk \left( [\HH_{\mm}^{1}(R/L)]_\nu \right) \\
			     & = \dim_\kk \left( [R/L]_\nu \right) +\dim_\kk \left( [\Ext_R^{r}(R/L,R)]_{-\nu-r-1} \right) \\
			     & = \sum_{i=1}^n \binom{\nu+\lambda_i-i}{\nu-i+1} + \dim_\kk \left( \left[R/(x_0,\dots,x_{r-2},x_{r-1}^{a_{1}})(a_1 + \cdots + a_{r} + r-1)\right]_{-\nu-r-1} \right) \\
			     & = \sum_{i=1}^n \binom{\nu+\lambda_i-i}{\nu-i+1} +
			     		\binom{a_{1} + \cdots + a_{r} - \nu-1}{1} - \binom{a_{2} + \cdots +a_{r} -\nu-1}{1}.
\end{align*}	 
The formula for $\dim_\kk \left([R/L]_\nu\right)$ can be found in \cite[Lemma 3.3]{SMOOTH_HILB}.
\end{proof}

Before we can prove the main result of this section, we need a simple lemma that relates the cohomologies of $V(fI)$ to those of $V(I)$ for any subscheme $V(I)\subseteq \PP_\kk^r$ of codimension at least two.

\begin{lemma} \label{lemma_detach} Let $\lambda = (\underbrace{r,\dots,r}_{a_r \emph{-times }}, \lambda')$ be an integer partition with $a_r > 0$ and $[I] \in \Hilb^{P_\lambda}_{\PP_{\kk}^r}$. 
Then we have $I = fI'$ with $[I'] \in \Hilb^{P_{\lambda'}}_{\PP^r_{\kk}}$, $\deg(f) = a_r$ and
$$
\dim_\kk \left( \HH^i(\PP_\kk^r,\OO_{V(I)}(\nu))\right) = \begin{cases} 
	\dim_\kk \left( \HH^i(\PP_\kk^r,\OO_{V(I')}(\nu-a_r))\right)  & \text{ if }  i \ne r-1  \\
	\binom{a_r-\nu-1}{r} - \binom{-\nu-1}{r}  & \text{ if } i = r-1.
	\end{cases}
$$
\end{lemma}
\begin{proof} The first statement can be found in \cite[Lemma 2.1]{TWO_BOREL}. The second statement follows from the local duality theorem and \cite[Fact 1]{REEVES_STILLMAN}, similar to \autoref{lemma_lex_ext}.
\end{proof}

The next proposition provides an equality between the fiber-full scheme and the Hilbert scheme when the latter is smooth. 

\begin{proposition} 
	\label{prop_smooth_Hilb}
	Let $\lambda$ denote an integer partition for which $\Hilb^{P_{\lambda}}_{\PP_\kk^r}$ is smooth. 
	Then we have the equality
	$$
	\mathrm{Fib}^{\mathbf{h}_{\lambda}}_{\PP_\kk^r} \,=\, \Hilb^{P_{\lambda}}_{\PP_\kk^r}.
	$$
\end{proposition}
\begin{proof} 
Since the Hilbert scheme $\Hilb^{P_{\lambda}}_{\PP_\kk^r}$ is smooth, it suffices to just check that $\fFib^{\mathbf{h}_{\lambda}}_{\PP_\kk^r}(\Spec(\kk)) = \fHilb^{P_{\lambda}}_{\PP_\kk^r}(\Spec(\kk))$. 
By \cite[Theorem A]{SMOOTH_HILB} there are seven different families of $\lambda$ for which the Hilbert scheme is smooth. We can reduce to considering partitions that satisfy $a_r = 0$, i.e., Hilbert schemes parametrizing subschemes of codimension at least two. Indeed, if $a_r > 0$, \autoref{lemma_detach} implies that  
$\Fib^{\mathbf{h}_{\lambda}}_{\PP_\kk^r} = \Hilb^{P_{\lambda}}_{\PP_\kk^r}$ if and only if 
$\Fib^{\mathbf{h}_{\lambda'}}_{\PP_\kk^r} = \Hilb^{P_{\lambda'}}_{\PP_\kk^r}$ where 
$\lambda = (\underbrace{r,\dots,r}_{a_r \text{-times }}, \lambda')$. 
Thus, for the rest of the proof we will only study those partitions in \cite[Theorem A]{SMOOTH_HILB} for which $a_r = 0$.

The conclusion is immediate for Case (7) as the Hilbert scheme consists of a single point. Case (1) corresponds to the Hilbert scheme of points in $\PP_\kk^2$. 
In this case  $\lambda = (1,\dots,1)$, equivalently $P_\lambda$ is constant, and this is covered by \autoref{0_DIM}.  Similarly, Case (6) reduces to $\lambda = (1,1,1)$ which is also covered by \autoref{0_DIM}.

To deal with Case (2) and Case (3) we use the fact that they have a unique Borel-fixed point. Let $\lambda$ be as in Case (2) or Case (3) and let $[I] \in \Hilb^{P_\lambda}_{\PP_\kk^r}$ with $I$ saturated. Since $\mathrm{gin}(I)$ is Borel-fixed \cite[Theorem 15.20]{EISEN_COMM}, we have $\mathrm{gin}(I) = L(\lambda)$. This implies $I$ and $L(\lambda)$ have the same Hilbert function and thus, $L(\lambda)$ is the lexicographic ideal associated to $I$. The result now follows from \cite[Theorem 0.1]{SBARRA}. The characteristic assumption of \cite{SBARRA} does not pose any issue because, in our case, the generic initial ideal is strongly stable \cite[proof of Theorem 0.1, page 274]{SBARRA}.

Case (4) and Case (5) correspond to Hilbert schemes with two Borel-fixed points. By \cite[Theorem A]{TWO_BOREL} we have two cases
\begin{itemize}
\item $\lambda = ((d+1)^q,1)$ with $d \geq 2$ and $q \geq 2$: The general member of $\Hilb^{P_\lambda}_{\PP_\kk^r}$ parametrizes $C \cup \{P\}$ where $C\subseteq \PP_\kk^{d+1}$ is a hypersurface of degree $q$ and $P$ is a point.
\item  $\lambda = (2^q,1)$ with $q \geq 4$: The general member of $\Hilb^{P_\lambda}_{\PP_\kk^r}$ parametrizes $C \cup P$ where $C$ is a plane curve of degree $q$ and $P$ is a point.
\end{itemize}

In either case, for any subscheme $[X] \in \Hilb^{P_\lambda}_{\PP_\kk^r}$, we have $\mathcal{I}_X = \mathcal{I}_{X'} \cap \mathcal{J}$ with $[X'] \in \Hilb^{P_{\lambda}-1}_{\PP_\kk^r}$ and $\mathcal{J}$ defining a, possibly embedded, $0$-dimensional subscheme. 
Arguing as in \autoref{TWISTED_CUBIC}, we see that $\text{Fib}^{P_\lambda}_{\PP_\kk^r} = \text{Hilb}^{P_\lambda}_{\PP_\kk^r} $ if and only if $\text{Fib}^{P_\lambda-1}_{\PP_\kk^r} = \Hilb^{P_\lambda-1}_{\PP_\kk^r}$.  
But the latter equality has already been established since $\lambda = ((d+1)^q)$ and $\lambda = (2^q)$, for the aforementioned $d,q$, have a unique Borel-fixed point. 
\end{proof}

\section{Square-free Gr\"obner degenerations and the fiber-full scheme}  \label{section_grobner}

Our results in this short section hint that the fiber-full scheme is the appropriate parameter space to study degenerations into a square-free monomial ideal (a line of research that we plan to address in the subsequent paper \cite{LOCAL_FIB_FULL}). 
We now use the following setup.

\begin{setup}
	\label{setup_Grobner_deform}
	Let $\kk$ be a field, $S = \kk[x_0,\ldots,x_r]$ be a polynomial ring, $\mm = (x_0,\ldots,x_r) \subset S$ be the maximal irrelevant ideal and $>$ be a monomial order on $S$.
	Let $A = \kk[t]$ be a polynomial ring and
	let $R = A \otimes_\kk S \cong S[t]$ be a polynomial ring over $S$.
\end{setup}

We use the notations and conventions of \cite[Chapter 15]{EISEN_COMM} for monomial orders and Gr\"obner bases.
Any element $f \in S$ can be uniquely written as 
$$
f \; = \; \sum_{j} \lambda_j \, \xx^{\alpha_j} 
$$
where $\lambda_j \in \kk$ and $\xx^{\alpha_j} = x_1^{\alpha_{j,0}} \cdots x_r^{\alpha_{j,r}} \in S$ is a monomial of $S$.
We uniquely identify the monomial $\xx^\alpha = x_0^{\alpha_{0}}  \cdots x_r^{\alpha_{r}} \in S$ with the integral vector  
$$
\mu(\xx^\alpha) \; := \; \left( \alpha_{0}, \ldots, \alpha_r \right) \in \NN^{r+1}.
$$
Let $\omega = (\omega_0, \ldots, \omega_r)\in \ZZ_+^{r+1}$ be a weight vector. 
The corresponding $\omega$-degree of the monomial $\xx^\alpha = x_0^{\alpha_{0}} \cdots x_r^{\alpha_{r}} \in S$ is given by 
$$
\deg_\omega(\xx^\alpha) \; := \;  \mu(\xx^\alpha ) \cdot \omega \;=\;  \alpha_{0} \omega_{0} + \cdots + \alpha_r \omega_{r}. 
$$
For an element $f \in S$, $\deg_\omega(f)$ is the maximum $\omega$-degree of the terms of $f$ and  $\iniTerm_\omega(f)$ is the sum of all the terms of $f$ of maximal $\omega$-degree.
For an ideal $I \subset S$, we define $\iniTerm_\omega(I) \subset S$ as  the ideal generated by $\iniTerm_\omega(f)$ for all $f \in I$.
For an element $f =  \sum_{j} \lambda_j \, \xx^{\alpha_j}  \in S$, the corresponding $\omega$-homogenization is given by
$$
\hom_\omega(f) \; := \; \sum_{j} \lambda_j \, \xx^{\alpha_j} \, t^{{\deg_\omega(f)} - {\deg_\omega(\xx^{\alpha_j})}}  \; \in  R. 
$$
For an ideal $I \subset S$, we define $\hom_\omega(I) \subset R$ as  the ideal generated by $\hom_\omega(f)$ for all $f \in I$.
For completeness, we recall the following well-known result.

\begin{proposition}[{see~ \cite[Theorem 15.17]{EISEN_COMM}}]	
	\label{thm_grob_mod}
	Let $I \subset S$ be an ideal.
	Then there exists a weight vector $\omega \in \ZZ_+^{r+1}$ such that the following statements hold:
	\begin{enumerate}[\rm (i)]
		\item $\iniTerm_\omega(I) = \iniTerm_>(I)$.
		\item $R/\hom_\omega(I)$ is a free $A$-module.
		\item $R/\hom_\omega(I) \otimes_{A} A/(t) \cong S/\iniTerm_>(I)$ and $R/\hom_\omega(I) \otimes_{A} \kk[t,t^{-1}] \cong S/I \otimes_{A} \kk[t,t^{-1}]$.
	\end{enumerate} 
\end{proposition}

We now state the main theorem pertaining to square-free Gr\"obner degenerations.

\begin{theorem}[Conca-Varbaro]
	\label{thm_sqr_free_Grob}
	Let $I \subset S$ be a homogeneous ideal. 
	If $\iniTerm_>(I)$ is square-free, then we have the equality 
	$$
	\dim_\kk\left(\big[\HL^i(S/I)\big]_\nu\right) \; = \; \dim_\kk\left(\big[\HL^i(S/\iniTerm_>(I))\big]_\nu\right)
	$$
	for all $i \ge 0$ and $\nu \in \ZZ$.
\end{theorem}
\begin{proof}
	This is \cite[Theorem 1.2]{CONCA_VARBARO}.
	Alternative subsequent proofs can be found in \cite[Corollary 3.5]{CONFERENCE_LEVICO} and \cite[Theorem 4.1]{FIBER_FULL}.
\end{proof}

Finally, we immediately obtain the following corollary.

\begin{corollary} \label{cor_sqr_free_curve}
	Let $I \subset S$ be a homogeneous ideal and assume that $\iniTerm_>(I)$ is square-free.
	Let $\hom_\omega(I) \subset R$ be a homogenization with special fiber equal to $\iniTerm_>(I)$.
	Let $Z \subset \PP_\kk^r$ be the closed subscheme given by $Z = \Proj(S/I)$.
	Let $\mathbf{h} = (h_0,\ldots,h_{r}) : \ZZ^{r+1} \rightarrow \NN^{r+1}$ be the tuple of functions given by $h_i(\nu) := \dim_\kk \left(\HH^i(Z, \OO_Z(\nu))\right)$.
	 For each $\alpha \in \kk$, let $Z_\alpha \subset \PP_{\kk}^r$ be the closed subscheme given by  
	 $$
	 Z_\alpha = \Proj\left(R/\hom_\omega(I) \otimes_{\kk[t]} \kk[t]/(t-\alpha)\right).
	 $$
	 Then we have that 
	 $$
	 Z_\alpha \;\,\text{ corresponds with a point in }\;\, \Fib_{\PP_\kk^{r}/\kk}^\mathbf{h}
	 $$
	 for all $\alpha \in \kk$.
\end{corollary}
\begin{proof}
	This is a direct consequence of \autoref{thm_grob_mod} and \autoref{thm_sqr_free_Grob} in terms of the fiber-full scheme.
\end{proof}

\section*{Acknowledgments}

We are grateful to Bernd Sturmfels for suggesting us to collaborate and for initiating our conversations.
We thank the referee for several useful comments.

\bibliographystyle{elsarticle-num} 
\bibliography{references}

\end{document}